\documentclass[a4paper,french,twoside,12pt]{smfartVS}

\usepackage{smfenum,smfthm,amsmath,amssymb,amscd}
\usepackage{mathrsfs,euscript,upgreek,color}
\usepackage[latin1]{inputenc}
\input{xypic}

\numberwithin{equation}{section} 
\theoremstyle{plain}

\def\AA{\mathbf{A}}
\def\CC{\mathbf{C}}
\def\FF{\mathbf{F}}
\def\NN{\mathbf{N}}
\def\QQ{\mathbf{Q}}

\def\ZZ{\mathbf{Z}} 

\def\St{{\rm St}}

\def\scusp{{\rm scusp}}

\def\seg{{\rm Seg}}

\def\A{{\rm A}}
\def\B{{\rm B}}

\def\D{{\rm D}}
\def\E{{\rm E}}
\def\F{{\rm F}}
\def\G{{\rm G}}

\def\I{{\rm I}}

\def\K{{\rm K}}
\def\L{{\rm L}}
\def\M{{\rm M}}
\def\N{{\rm N}}
\def\P{{\rm P}}

\def\R{{\rm R}}

\def\U{{\rm U}}
\def\V{{\rm V}}

\def\X{{\rm X}}

\def\Z{{\rm Z}}

\def\Gg{\mathscr{G}}

\def\Jj{\mathscr{J}}

\def\Oo{\mathscr{O}}

\def\a{\alpha} 
\def\b{\beta}

\def\p{\mathfrak{p}}
\def\s{\sigma}

\def\ie{c'est-à-dire }

\def\rp{\rangle}
\def\>{\geqslant}
\def\<{\leqslant}

\def\Aut{{\rm Aut}}
\def\Mat{\mathscr{M}}
\def\GL{{\rm GL}}

\def\mult#1{{#1}^{\times}}

\def\AA{\mathfrak{A}}

\def\Oo{\EuScript{O}}

\def\CR{{\rm R}}

\def\supp{{\rm supp}}
\def\cusp{{\rm cusp}}
\def\Irr{{\rm Irr}}

\def\ll{\mathfrak{l}}
\def\mm{\mathfrak{m}}

\def\sss{\mathfrak{s}}

\def\tt{\Xi}

\def\r{{\bf r}}
\def\ip{\boldsymbol{i}}

\def\rp{\boldsymbol{r}}

\def\Cusp{\EuScript{C}}
\def\Scusp{\EuScript{S}}
\def\SCusp{\Scusp}

\def\sy#1{\boldsymbol{[}#1\boldsymbol{]}}

\def\cC#1{{\rm supp}(#1)}
\def\cCO#1{{\rm supp}^0(#1)}
\def\sc#1{{\rm sc}(#1)}
\def\iso#1{\smash{\mathop{\longrightarrow}\limits^{#1}}}
\def\widetild#1{#1^{\vee}}

\def\({\left(}
\def\){\right)}

\def\Div{\ZZ}
\def\Dive{\NN}

\def\m{\mm}

\def\St{{\rm St}}

\def\seg{{\rm Seg}}

\def\qlb{\overline{\QQ}_{\ell}}
\def\zlb{\overline{\ZZ}_{\ell}}
\def\flb{\overline{\FF}_{\ell}}

\def\vr{\varrho}

\def\MS{{\rm Mult}}

\def\1{{\bf 1}}

\newcounter{nonum}

\author{Alberto M\'\i nguez}
\address{Institut de Math\'ematiques de Jussieu, 
Universit\'e Pierre et Marie Curie,
4, place Jussieu. 75005 Paris, France.} 
\urladdr{http://www.math.jussieu.fr/$\sim$minguez/}
\email{minguez@math.jussieu.fr}

\author{Vincent Sécherre}
\address{Université de Versailles Saint-Quentin-en-Yvelines\\
Laboratoire de Mathématiques de Versailles\\
45 avenue des Etats-Unis\\
78035 Versailles cedex, France}
\email{vincent.secherre@math.uvsq.fr}

\title[Représentations banales de $\GL_{m}(\D)$]
{Représentations banales de $\GL_{m}(\D)$}

\begin{altabstract}
Let $\F$ be a non-Archimedean locally compact field of residue 
characteristic $p$, let $\D$ be a finite dimensional central 
division $\F$-algebra and let $\CR$ be an algebraically closed field 
of characteristic different from $p$.  
We define \textit{banal} irreducible $\CR$-repre\-sentation of 
$\G=\GL_{m}(\D)$. 
This notion 
involves a condition on the cuspidal support of the repre\-sentation 
depending on the characteristic of $\CR$. 
When this characteristic is banal with respect to $\G$, 
in particular when $\CR$ is 
the field of complex numbers, any irreducible $\CR$-representation of $\G$ 
is banal. 
In this article, we give a classification of all banal ir\-re\-du\-ci\-ble 
$\CR$-representations of $\G$ in terms of certain multisegments, called 
banal. 
When $\CR$ is the field of complex numbers, our method provides a 
new proof, entirely local, of Tadi\'c's classification of irreducible complex 
smooth representations of $\G$. 
\end{altabstract}

\thanks{Ce travail a bénéficié de financements de l'EPSRC 
(GR/T21714/01, EP/G001480/1) et l'Agence Natio\-nale de la Recherche 
(ANR-08-BLAN-0259-01, ANR-10-BLANC-0114).
Le pre\-mier auteur est financé en partie par MTM2010-19298 
et FEDER}

\begin{document}

\maketitle

\section{Introduction}

Soit $\F$ un corps commutatif localement compact non archimédien de 
ca\-rac\-téristique ré\-si\-duelle $p$ et soit $\D$ une $\F$-algèbre à division 
centrale de dimension finie dont le degré réduit est noté $d$.  
Pour tout entier $m\>1$, on dé\-si\-gne par 
$\G_{m}$ le groupe $\GL_{m}(\D)$, qui est une forme intérieure du groupe 
linéaire $\GL_{md}(\F)$. 
Les représentations lisses irréductibles complexes du groupe 
$\GL_{md}(\F)$ ont été 
classées par Zelevinski \cite{Ze2} en  termes de paramètres appelés
\textit{multisegments}.
Dans \cite{Tadic}, où la caractéristique de $\F$ est supposée nulle, 
Tadi\'c donne une classification des représentations lisses 
irréductibles complexes de $\G_m$ en termes de multi\-segments.
La méthode utilisée par Tadi\'c s'appuie sur les résultats 
de \cite{DKV} (qui reposent eux-mêmes sur la formule des traces simple) 
et en particulier sur la correspondance de Jacquet-Langlands et la 
classification des représentations tempérées en fonction de la série 
discrète (voir {\it ibid.}, théorème B.2.d). 
Dans \cite{Bad}, Badulescu étend ces deux résultats au cas où 
$\F$ est de caractéristique $p$, et on trouve dans \cite{BHLS} la 
classification des représentations lis\-ses irréductibles 
complexes de $\G_m$ sans restriction sur la caractéristique de $\F$. 

\medskip

Dans cet article, on s'intéresse au problème de la classification des 
représentations lisses irréductibles de $\G_m$ à coefficients dans un corps 
$\CR$ algébriquement clos de caractéristique $\ell$ différente de $p$, 
appelées aussi représentations (lisses irréductibles) modulaires
dans le cas où $\ell$ est non nulle. 
Dans le cas modulaire, 
on ne dispose pas d'une formule des traces et le théo\-rè\-me du
quotient de Langlands, qui permet dans la construction de Tadi\'c d'obtenir 
toutes les représentations irréductibles en fonction des représentations 
tempérées, n'est pas valable. 
Il faut trouver une approche différente de celle employée dans le cas 
complexe.

\medskip

Dans \cite{MS}, nous avons effectué la classification complète des
$\CR$-représentations irré\-duc\-ti\-bles de $\G_m$ en 
termes de multisegments.  
C'est un long travail dont l'un des principaux outils est 
la théorie des types, et qui s'appuie sur des résultats pro\-fonds de 
\cite{ArikiBook,CG} sur les représentations des algèbres de Hecke 
affines de type $\A$ en une raci\-ne de l'unité. 

\medskip

Dans cet article, nous proposons une classification, 
\textit{indépendante de celle obtenue dans \cite{MS}
et ne s'appuyant pas sur \cite{ArikiBook,CG}}, 
de certaines $\CR$-représentations irréductibles de 
$\G_m$ que nous appelons \textit{banales}. 
Une représentation irréductible de $\G_m$ est banale 
si son support cuspidal satisfait à une condition technique 
qui dépend de la carac\-téristique de $\CR$ 
(voir plus bas pour une définition précise).  
Si cette caractéristique $\ell$ est banale, \ie si $\ell$ est différent de $p$ 
et ne divise aucun des entiers $q^{di}-1$ pour $i\in\{1,\dots,m\}$, où $q$ 
désigne le cardinal du corps résiduel de $\F$, alors toute 
$\CR$-représentation irréductible de $\G_m$ est banale.  
Ceci se produit en particulier si $\CR$ est le corps des nombres 
complexes~: dans ce cas, notre article fournit une nouvelle preuve de la
clas\-si\-fi\-cation de Tadi\'c, 
indépendante de la caractéristique de $\F$ et purement locale~: 
on n'uti\-lise ni la formule des traces, ni la correspondance de 
Jacquet-Langlands.  

\medskip

Nous avons choisi de publier la classification des représentations 
banales indé\-pen\-dam\-ment de la classification complète de \cite{MS}, 
d'une part parce que l'approche employée
dans \cite{MS} se simplifie de façon remarquable dans le cas banal 
(en particulier nous n'avons pas besoin ici des résultats profonds de 
\cite{ArikiBook,CG}), 
d'autre part pour insister sur le fait que notre approche fournit une preuve 
purement locale de la classification de Tadi\'c. 

\medskip

L'une des principales difficultés auxquelles on est confronté lorsqu'on étudie 
les re\-pré\-sen\-ta\-tions mo\-du\-lai\-res de $\G_m$ est l'apparition de 
représentations cuspidales non supercuspidales (voir \cite{Vig1}).  
C'est ce qu'on évite ici en se restreignant aux représentations banales~: 
une représentation ir\-ré\-duc\-ti\-ble banale est cuspidale si et seulement 
si elle est supercuspi\-dale.  
Aussi certaines techniques algébriques de la théorie des représentations 
complexes em\-ployées par Zelevinski s'étendent-elles au cas banal. 
On obtient une double classification (à la Zelevinski et à la Langlands)
plus précise que celle de \cite{MS} (qui ne donne qu'une classification à la 
Zelevinski) et la preuve en est beaucoup plus simple.

\medskip

Il est temps déjà de donner plus de détails sur le contenu de cet article.
Etant donnée une représentation irréductible cuspidale $\rho$ de $\G_m$, 
on lui associe
dans \cite[\textsection 7.1]{MS} un carac\-tère non ramifié $\nu_{\rho}$ tel que, 
pour toute représentation cuspidale $\rho'$ de $\G_{m'}$, $m'\>1$, 
l'in\-duite parabolique normalisée $\rho\times\rho'$ soit réductible si et seulement si 
$\rho'$ est isomorphe à $\rho\nu_\rho^{}$ ou à $\rho\nu_{\rho}^{-1}$.  
Par exemple, si $\D$ est égale à $\F$, le caractère $\nu_{\rho}$ est indépendant de $\rho$ 
et est égal à $|\det |_{\F}$, où $|\ |_{\F}$ désigne la valeur absolue normalisée de
$\F$.  
On note $\ZZ_\rho$ l'ensemble des classes d'isomorphisme des
$\rho\nu_{\rho}^{i}$ pour $i\in\ZZ$.  
Un {\it segment} est une suite finie de la forme~:
\begin{equation*}
\(\rho\nu_{\rho}^a,\rho\nu_{\rho}^{a+1},\dots,\rho\nu_{\rho}^{b}\)
\end{equation*}
où $a,b$ sont des entiers tels que $a\<b$.  
Un tel segment est noté $\left[a,b\right]_\rho$.  
On a une notion naturelle d'équivalence entre segments, et on définit un 
{\it multisegment} comme une somme formelle de clas\-ses d'équi\-va\-len\-ce de 
segments, \ie un élément du groupe abélien libre engendré par les classes 
d'équi\-va\-len\-ce de segments.  
Le {\it support} d'un segment $\left[a,b\right]_\rho$ est la somme 
formelle des classes d'isomorphisme des $\rho\nu_{\rho}^i$ pour $a\<i\<b$, 
somme qui ne dépend que de la classe de ce segment, et le support 
d'un multisegment est la somme des supports des segments qui le 
composent.  

\begin{defi}
\begin{enumerate}
\item 
Une somme formelle $\rho_1+\dots+\rho_r$ 
de classes de représenta\-tions irréductibles cuspidales est dite
\textit{banale} si, pour toute représentation irréductible 
cuspi\-da\-le $\rho$, il existe 
un élément de $\ZZ_\rho$ qui n'apparaît pas dans cette somme. 
\item
Un multisegment est dit \textit{banal} si son support est banal~; 
une représentation irré\-duc\-tible est dite \textit{banale} si son support 
cuspidal est banal.  
\end{enumerate}
\end{defi}

La propriété suivante, prouvée dans \cite[\textsection 8.2]{MS}, justifie l'importance des 
représentations banales. 

\begin{prop}
Soient $\rho_1,\dots,\rho_r$ des représentations cuspidales, avec $r\>2$.  
Supposons que la somme formelle $\rho_1+\dots+\rho_r$ soit banale.  
Alors l'induite parabolique~: 
\begin{equation*}
\rho_1\times\dots\times\rho_r
\end{equation*}
ne contient aucun sous-quotient irréductible cuspidal.
\end{prop}

On définit dans \cite[\textsection 8.3]{MS} le support supercuspidal d'une représentation 
irréductible de $\G_m$. 
On montre 
que le support supercuspidal d'une représentation banale est égal 
à son support cuspidal.  

\medskip

Nous pouvons maintenant énoncer les théorèmes de classification 
contenus dans cet article.  
Soit $\Delta=\left[a,b\right]_\rho$ un segment banal.  
Alors la représentation induite~: 
\begin{equation*}
\rho\nu_\rho^a\times\dots\times\rho\nu_\rho^b
\end{equation*}
possède une unique sous-re\-pré\-sen\-ta\-tion ir\-ré\-ductible, 
notée $\Z(\Delta)$, et un unique quotient ir\-ré\-ductible, noté 
$\L(\Delta)$ (voir la proposition \ref{segm}).  
Le résultat principal de cet article est le double théorème de classification 
suivant.  

\begin{theo}
\label{Z21}
\begin{enumerate}
\item 
Soient $\Delta_1,\dots,\Delta_r$ des segments tels que, 
pour tous $i<j$, le segment $\Delta_i$ ne 
précède pas $\Delta_j$
et tels que le multisegment $\m=\Delta_1+\dots+\Delta_r$ 
soit banal. 
Alors~: 
\begin{equation}
\label{ZDi}
\Z(\Delta_1)\times \dots \times \Z(\Delta_r)
\end{equation}
admet une unique sous-représentation irréductible, ne dépendant 
que de $\m$ et notée $\Z(\m)$.
Elle est banale et sa mul\-ti\-pli\-cité comme sous-quotient de (\ref{ZDi})
est égale à $1$. 
\item 
L'application $\m\mapsto\Z(\m)$ définit une bijection entre multisegments 
banals de degré $n$ et classes d'isomorphisme de représentations irréductibles 
banales de $\G_n$.
\end{enumerate}
\end{theo}

\begin{theo}
\label{Z22}
\begin{enumerate}
\item 
Soient $\Delta_1,\dots,\Delta_r$ des segments tels que, 
pour tous $i<j$, le segment $\Delta_i$ ne 
précède pas $\Delta_j$
et tels que le multisegment $\m=\Delta_1+\dots+\Delta_r$ 
soit banal. 
Alors~: 
\begin{equation}
\label{LDi}
\L(\Delta_1)\times \dots \times \L(\Delta_r)
\end{equation}
admet un unique quotient irréductible, ne dépendant 
que de $\m$ et noté $\L(\m)$.
Il est banal et sa mul\-ti\-pli\-cité comme facteur de (\ref{LDi})
est égale à $1$. 
\item 
L'application $\m\mapsto\L(\m)$ définit une bijection entre multisegments 
banals de degré $n$ et classes d'isomorphisme de représentations irréductibles 
banales de $\G_n$.
\end{enumerate}
\end{theo}

La première classification est dite \textit{à la Zelevinski} et la seconde 
\textit{à la Langlands}.  
Pour prouver les parties (1) des théorèmes et l'injectivité des
classifications, on
adapte les idées du théorème du quotient de Langlands au cas
modulaire.  
Remarquons que la condition de banalité nous 
permet d'ordonner tout multisegment banal de sorte que pour tous 
$i<j$, le segment $\Delta_i$ ne précède pas $\Delta_j$. 
Pour prouver la surjectivité des classifications, on procède comme dans 
\cite{Ze2} ~: on utilise le fait que 
$\Z(\Delta)\times \Z(\Delta' )$ et $ \L(\Delta)\times \L(\Delta')$ 
sont de longueur $2$ ou $1$, selon que les segments $\Delta$ et $\Delta'$ 
sont liés ou non 
(ce qui est faux si on n'impose pas la condition de banalité).  

\medskip

Supposons maintenant que $\CR$ soit le corps des nombres complexes.  
Dans le cas où $\D$ est égale à $\F$, notre preuve, purement combinatoire, 
est plus simple que celle de Zelevinski \cite{Ze2}, qui utilise à plusieurs 
reprises des arguments de \cite{BZ2}.  
Notre preuve ne fait presque aucune diffé\-rence entre les deux classifications.  
Dans le le cas où $\D$ n'est pas commutative, le théorème \ref{Z22} est prouvé dans 
\cite{Tadic} pour $\F$ de caractéristique nulle et dans \cite{BHLS} 
pour $\F$ de carac\-téristique $p$. 
Le théorème \ref{Z21} est nouveau.  

\medskip

Pour finir, dans la section \ref{rele}, on montre que toute
$\flb$-représentation irréductible banale se relève en une 
$\qlb$-représentation irréductible entière.  
En général, une $\flb$-re\-pré\-sentation non banale 
ne peut pas toujours être relevée (voir \cite[\textsection 9.7]{MS}). 

\section*{Remerciements}

La majeure partie de ce travail a été effectuée alors que le second 
auteur était en poste à l'Université de la Méditerranée et à 
l'Institut de Mathématiques de Luminy.
Le début de  ce travail faisait partie  de la thèse du premier  auteur sous la
direction de G. Henniart. Les auteurs le remercient pour son support constant 
et ses nombreuses suggestions. 

\section*{Notations et conventions}

{\bf 1.}
\textit{Dans tout cet article}, $\F$ est un corps commutatif localement compact non 
archimé\-dien, de carac\-téristique résiduelle notée $p$, 
et $\CR$ est un corps algébriquement clos de carac\-téristique différente de $p$. 

\medskip

{\bf 2.}
Une $\F$-\emph{algèbre à division} est une $\F$-algèbre 
centrale de dimension finie dont l'anneau sous-jacent est un corps qui 
n'est pas né\-ces\-sai\-re\-ment commutatif. 
Si $\K$ est une extension finie de $\F$, ou plus généralement 
une algèbre à division sur une extension finie de $\F$, 
on note $\Oo_\K$ son anneau d'entiers, $\p_\K$ son idéal maximal 
et $q_\K$ le cardinal de son corps résiduel. 
On pose enfin $q=q_\F$. 

\medskip

{\bf 3.}
Une $\CR$-{\it re\-pré\-sen\-ta\-tion lisse} d'un groupe topologique 
$\G$ est un couple composé d'un $\R$-espace vectoriel $\V$ et d'un 
homomorphisme de groupes de $\G$ dans $\Aut_\R(\V)$ tel que le stabilisateur 
de tout vecteur de $\V$ soit un sous-groupe ouvert de $\G$. 
{\it Dans cet article, toutes les repré\-sentations sont supposées lisses.}

Un $\CR$-\textit{caractère} de $\G$ est un homomorphisme 
de groupes de $\G$ dans $\mult\CR$ de noyau ouvert.  
Si $\pi$ est une $\CR$-représentation de $\G$, on désigne par $\pi^{\vee}$ 
sa contragrédiente. 
Si en outre $\chi$ est un $\CR$-caractère de $\G$, on note $\chi\pi$ ou 
$\pi\chi$ la représentation tordue $g\mapsto\chi(g)\pi(g)$.

S'il n'y a pas d'ambiguïté, on écrira \textit{caractère} et 
\textit{re\-présentation} plutôt que $\CR$-caractère et 
$\CR$-repré\-sen\-ta\-tion. 

\medskip

{\bf 4.}
Étant donné un ensemble $\X$, on note $\Div(\X)$ le groupe 
abélien libre de base $\X$ constitué des applications de $\X$ dans $\ZZ$ 
à support fini et $\Dive(\X)$ le sous-ensemble de $\Div(\X)$ 
constitué des applications à valeurs dans $\NN$. 
Pour $f,g\in\Div(\X)$, on note $f\<g$ si $g-f\in\Dive(\X)$, 
ce qui définit une relation d'ordre partiel sur $\Div(\X)$.

\section{Préliminaires}
\label{Banalite}

Dans tout ce qui suit, on fixe une $\F$-algèbre à division $\D$ 
de degré réduit noté $d$.
Pour tout entier $m\>1$, on dé\-si\-gne par $\Mat_{m}(\D)$ la 
$\F$-algèbre des matrices de taille $m\times m$ à coefficients dans $\D$ 
 et on pose $\G_{m}=\GL_{m}(\D)$.  
Il est commode de convenir que $\G_{0}$ est le groupe trivial. 

\subsection{}
\label{DefNu}

Soit $m\>1$, et soit $\N_{m}$ la norme réduite de $\Mat_{m}(\D)$ sur $\F$.  
On note $|\ |_{\F}$ 
la valeur absolue normalisée de $\F$, \ie la valeur absolue 
don\-nant à une uniformisante de $\F$ la valeur $q^{-1}$.
Puisque l'image de $q$ dans $\CR$ est in\-ver\-si\-ble, 
cette valeur absolue définit un $\CR$-caractère de $\mult\F$ noté $|\ |_{\F,\CR}$.
L'ap\-pli\-ca\-tion $g\mapsto|\N_{m}(g)|_{\F,\CR}$ est un 
$\CR$-ca\-rac\-tè\-re de $\G_{m}$, qu'on notera simplement $\nu$.  

\subsection{}
\label{UZero}

Pour $m\>0$, on note $\Irr_{\CR}(\G_m)$ l'ensemble des classes 
d'isomorphisme de 
$\CR$-re\-pré\-sen\-ta\-tions ir\-ré\-duc\-ti\-bles de $\G_m$ 
et $\Gg_{\CR}(\G_m)$ le groupe de Gro\-then\-dieck de ses 
$\CR$-re\-pré\-sen\-ta\-tions de longueur finie, 
qui est un $\ZZ$-module libre de base $\Irr_{\CR}(\G_m)$. 

Si $\pi$ est une $\CR$-représentation de longueur finie de $\G_m$, 
on note $\deg(\pi)=m$, qu'on ap\-pel\-le le {\it degré} de $\pi$, et 
on note $\sy{\pi}$ son image dans $\Gg_{\CR}(\G_m)$. 
En particulier, si $\pi$ est irréductible, $\sy{\pi}$ désigne 
sa classe d'isomorphisme. 
Lorsqu'aucune confusion ne sera possible, il nous arrivera d'identifier 
une $\CR$-représentation avec sa classe d'isomorphisme.

On désigne par $\Irr_{\CR}$ la réunion disjointe des ensembles 
$\Irr_{\CR}(\G_{m})$ pour $m\>0$, et par $\Gg_{\CR}$ la somme 
directe des $\Gg_{\CR}(\G_m)$ pour $m\>0$, qui est un $\ZZ$-module 
libre de base $\Irr_{\CR}$. 

\subsection{}
\label{GeEm}

Si $\a=(m_{1},\ldots,m_{r})$ est une famille 
d'entiers positifs ou nuls dont la somme est égale à $m$, 
il lui correspond le 
sous-groupe de Levi standard $\M_{\a}$ de $\G_{m}$ constitué des matrices 
diagonales par blocs de tailles $m_{1},\ldots,m_{r}$ respectivement, que 
l'on identifie naturellement au produit 
$\G_{m_{1}}\times\cdots\times\G_{m_{r}}$. 
On note $\P_{\a}$ 
le sous-groupe para\-bo\-li\-que de $\G_{m}$ de facteur de Levi 
$\M_{\a}$ formé des matrices tri\-an\-gu\-lai\-res supérieures 
par blocs de tailles $m_{1},\ldots,m_{r}$ respectivement, et on note 
$\U_{\a}$ 
son radical unipotent.

\textit{On choisit une fois pour toutes une racine carrée de $q$
dans $\CR$.}
On note $\rp_\a$ le foncteur de restriction 
parabolique nor\-ma\-li\-sé relativement à ce choix, 
et $\ip_\a$ son adjoint à droite, \ie le foncteur d'induction 
para\-bo\-li\-que nor\-ma\-li\-sé lui correspondant.
Ces foncteurs sont exacts,
et préservent l'admissibilité et le fait d'être de longueur finie.

Si, pour chaque $i\in\{1,\ldots,r\}$, on a une 
$\CR$-représentation $\pi_{i}$ de $\G_{m_i}$, on note~: 
\begin{equation}
\label{VentreDieu}
\pi_1\times\cdots\times\pi_r=\ip_{\a}(\pi_1\otimes\cdots\otimes\pi_r).
\end{equation}
Si les $\pi_{i}$ sont de longueur finie, 
la quantité $\sy{\pi_1\times\cdots\times\pi_r}$
ne dépend que de $\sy{\pi_{1}},\dots,\sy{\pi_{r}}$.
L'ap\-pli\-cation~:
\begin{equation}
\label{VentreDieuGris}
(\sy{\pi_1},\dots,\sy{\pi_r})\mapsto\sy{\pi_1\times\cdots\times\pi_r}
\end{equation}
induit par linéarité une application linéaire 
de $\Gg_{\CR}(\G_{m_1})\times\dots\times\Gg_{\CR}(\G_{m_r})$ dans 
$\Gg_{\CR}(\G_{m})$. 
Ceci munit $\Gg_{\CR}$ d'une structure de 
$\ZZ$-al\-gè\-bre commutative gra\-duée 
(voir \cite{Ze2,Tadic,BHLS} dans le cas complexe
et \cite{Dat3,MS} dans le cas modulaire).

On note également ${\rp}_{\a}^{-}$ le foncteur de restriction 
para\-bo\-li\-que relativement au sous-groupe parabolique opposé 
à $\P_{\a}$ relativement à $\M_\a$, 
\ie formé des matrices tri\-an\-gu\-lai\-res inférieures
par blocs de tailles $m_{1},\ldots,m_{r}$ respectivement, 
et on note $\a^{-}=(m_{r},\ldots,m_{1})$ la famille déduite 
de $\a$ en in\-ver\-sant l'ordre des termes. 

\subsection{}
\label{DefRepCusp}

Une $\CR$-représentation irréductible de $\G_m$ est dite {\it cus\-pi\-da\-le} 
si son image par $\rp_\a$ est nulle pour toute famille
$\a=(m_{1},\ldots,m_{r})$ d'entiers compris entre $0$ et $m-1$ et de somme 
$m$.
Elle est dite {\it super\-cus\-pi\-da\-le} si en outre, pour 
toute $\CR$-représentation irréductible $\pi_{i}$ 
de $\G_{m_i}$, sa classe d'isomorphisme 
n'apparaît pas dans $\sy{\pi_1\times\cdots\times\pi_r}$

On note $\Cusp_{\CR}(\G_m)$ et $\SCusp_{\CR}(\G_m)$ les sous-ensembles de 
$\Irr_{\CR}(\G_m)$ constitués respectivement des clas\-ses 
d'iso\-mor\-phis\-me 
de $\CR$-re\-pré\-sen\-ta\-tions ir\-ré\-duc\-ti\-bles cuspidales et 
supercuspidales de $\G_m$.
On note $\Cusp_{\CR}$ et $\SCusp_{\CR}$ la réunion disjointe de
$\Cusp_{\CR}(\G_{m})$ et $\SCusp_{\CR}(\G_m)$ respectivement, pour $m\>0$.

\subsection{}
\label{SuppCusp}

Une \textit{paire (super)cuspidale} de $\G_m$ est un couple $(\M,\vr)$ formé 
d'un sous-groupe de Levi $\M$ de $\G_m$ et d'une $\CR$-représentation
irréductible (super)cuspidale $\vr$ de $\M$.
Si $\M=\M_\a$ pour une famille $\a=(m_{1},\ldots,m_{r})$ d'entiers positifs 
de somme $m$, alors $\vr$ est de la forme $\rho_1\otimes\dots\otimes\rho_r$, 
où $\rho_i$ est une $\CR$-représentation irréductible (super)cuspidale 
de $\G_{m_i}$, pour $i\in\{1,\dots,r\}$. 

Si $\pi$ est une $\CR$-re\-pré\-sen\-ta\-tion irréductible de $\G_m$, il 
existe une famille $\a=(m_{1},\ldots,m_{r})$ et, pour chaque $i$, une 
$\CR$-représentation irréductible 
cuspidale (supercuspidale) 
$\rho_i$ de $\G_{m_i}$ telles que $\pi$ soit isomorphe à une 
sous-représentation (un sous-quotient)
de la représen\-tation induite $\rho_1\times\dots\times\rho_r$.  
La somme~: 
\begin{equation*}
\sy{\rho_1}+\dots+\sy{\rho_r}
\end{equation*}
dans $\Dive(\Cusp_{\CR})$ 
(resp. $\Dive(\SCusp_{\CR})$) est unique et s'appelle 
le sup\-port cuspidal (supercuspidal) de $\pi$ et est notée $\cusp(\pi)$ 
(resp. $\scusp(\pi)$).  
Pour une preuve de l'unicité du support cuspidal (supercuspidal), 
on renvoie à \cite{MS}. 

\subsection{}
\label{lemmegeo}

Dans ce paragraphe, on donne une version combinatoire du lemme géométrique
de Bernstein-Zelevinski (voir \cite[\textsection 2.4.3]{MS}).  
Soient $\a=(m_1,\dots,m_r)$ et $\b=(n_1,\dots,n_s)$ deux familles d'entiers 
de sommes toutes deux égales à $m\>1$.
Pour chaque $i\in\{1,\ldots,r\}$, soit $\pi_i$ une $\CR$-re\-pré\-sen\-tation 
ir\-ré\-duc\-ti\-ble de $\G_{m_i}$, et posons 
$\pi=\pi_1\otimes\cdots\otimes\pi_r\in\Irr_{\CR}(\M_{\a})$. 
On note $\Mat^{\a,\b}$ l'ensemble des ma\-tri\-ces $\B=(b_{i,j})$ composées
d'entiers positifs tels que~:
\begin{equation*}
\label{2} 
\sum_{j=1}^s b_{i,j}=m_i, \quad \sum_{i=1}^rb_{i,j}=n_j,
\quad i\in\{1,\ldots,r\}, \quad j\in\{1,\ldots,s\}.
\end{equation*}
Fixons $\B\in\Mat^{\a,\b}$ et notons $\a_i=(b_{i,1},\dots,b_{i,s})$ et 
$\b_j=(b_{1,j},\dots,b_{r,j})$, qui sont des partitions de $m_i$ et 
de $n_j$ respectivement. 
Pour chaque $i\in\{1,\ldots,r\}$, on écrit~:
\begin{equation*}
\s^{(k)}_{i}=\s^{(k)}_{i,1}\otimes\cdots\otimes\s^{(k)}_{i,s}, 
\quad
\s^{(k)}_{i,j}\in\Irr_{\CR}(\G_{b_{i,j}}),
\quad k\in\{1,\ldots,r_i\},
\end{equation*}
les différents facteurs de composition de $\rp_{\a_i}(\pi_i)$.
Pour tout $j\in\{1,\ldots,s\}$ et toute famille d'entiers
$(k_1,\ldots,k_r)$ tels que $1\<k_i\<r_i$, 
on définit une $\CR$-représentation $\sigma_{j}$ de $\G_{n_j}$ par~: 
\begin{equation*}
\s_{j}=\ip_{\b_j}
\(\s^{(k_1)}_{1,j}\otimes\cdots\otimes\s^{(k_r)}_{r,j}\). 
\end{equation*}
Alors les $\CR$-représentations~: 
\begin{equation*}
\label{Druon}
\s_{1}\otimes\cdots\otimes\s_{s}, 
\quad
\B\in\Mat^{\a,\b},
\quad
1\<k_i\<r_i,
\end{equation*}
forment une suite de com\-po\-si\-tion de
$\rp_{\b}(\ip_{\a}(\pi))$. 

\medskip

On reprend les notations ci-dessus avec $s=2$.
Le résultat suivant est prouvé dans \cite[Proposition 2.8]{MS}.  

\begin{prop}
\label{mmm1}
On suppose que, pour tout $i\in\{1,\ldots,r\}$, toute partition 
$\a_i$ de $m_i$ et tout facteur de composition 
$\s^{(k)}_{i}=\s^{(k)}_{i,1}\otimes\s^{(k)}_{i,2}$ de 
$\rp_{\a_i}(\pi_i)$, 
on ait~:
\begin{equation*}
\cusp(\s^{(k)}_{i,2})\nleqslant\sum\limits_{i<j\<r}\cusp(\pi_j).
\end{equation*}
Alors $\pi$ apparaît avec multiplicité $1$ dans 
$\sy{\rp_{\a}(\ip_{\a}(\pi))}$. 
Ainsi $\ip_{\a}(\pi)$ (resp.  $\ip_{\a^-}(\pi)$) a une unique 
sous-représentation (resp.  un unique quotient) 
irréductible, dont la multiplicité dans 
$\sy{\ip_{\a}(\pi)}$ (resp. $\sy{\ip_{\a^-}(\pi)}$) est égale à $1$. 
\end{prop}

\section{La théorie des segments}
\label{segments}

On rappelle les notions de segment et de multisegment (voir 
\cite[\textsection\textsection 7.3 et 9.1]{MS}).  

\subsection{}
\label{ccusp}

Soit $m\>1$ un entier, et soit 
$\rho$ une $\CR$-représentation irréductible cuspidale de $\G_m$.  
Dans \cite[\textsection 7.1]{MS}, on associe à $\rho$ 
un $\CR$-caractère non ramifié $\nu_{\rho}$ de $\G_m$ tel que, 
si $\rho'$ est une $\CR$-re\-pré\-sentation irréductible cuspidale de 
$\G_{m'}$ avec $m'\>1$, 
la représentation in\-duite 
$\rho\times\rho'$ soit réductible si et seulement si $m'=m$ et si 
$\rho'$ est isomorphe à 
$\rho\nu_{\rho}^{}$ ou à $\rho\nu_{\rho}^{-1}$.
Ce caractère ne dépend que de la classe d'inertie de $\rho$.  
On pose~:
\begin{equation}
\ZZ_\rho=\{\sy{\rho\nu_{\rho}^{i}}\ |\ i\in\ZZ\}.
\end{equation}
Dans le cas où $\CR$ est de caractéristique non nulle, cet ensemble est fini 
et on note $e(\rho)$ son cardinal.  
Si la caractéristique de $\CR$ est nulle, 
$\ZZ_\rho$ est un ensemble infini et on convient que $e(\rho)=+\infty$.

\subsection{}

Soit un entier $m\>1$, soit $\rho$ une $\CR$-re\-pré\-sen\-ta\-tion 
irréductible cuspidale de $\G_m$ et soient $a,b\in\ZZ$ des entiers 
tels que $a\<b$. 

\begin{defi}
Un {\it segment} est une suite finie de la forme~:
\begin{equation}
\(\rho\nu_{\rho}^a,\rho\nu_{\rho}^{a+1},\dots,\rho\nu_{\rho}^{b}\).
\end{equation}
Une telle famille est notée $\left[a,b \right]_\rho$. 
\end{defi}

Si $\Delta=\left[a,b\right]_{\rho}$ est un segment, on note~: 
\begin{equation}
n(\Delta)=b-a+1,
\quad 
\deg(\Delta)=(b-a+1)m,
\quad 
a(\Delta)=\rho\nu_{\rho}^a,
\quad
b(\Delta)=\rho\nu_{\rho}^b, 
\end{equation}
respectivement la longueur, le degré et les extrémités de $\Delta$.
Si $a+1\<b$, on pose~:
\begin{equation}
^-\Delta=\left[a+1,b\right]_\rho,
\quad
\Delta^-=\left[a,b-1\right]_\rho.
\end{equation}
Le {\it support} de $\Delta$, noté $\supp(\Delta)$, 
est l'élément de $\Dive(\Cusp_\CR)$ défini par~: 
\begin{equation}
\supp(\Delta)=\sy{\rho\nu_\rho^a}+\dots+\sy{\rho\nu_\rho^b}. 
\end{equation}
On note enfin~:
\begin{equation}
\label{ContraSeg}
\widetild{\Delta}=\left[-b,-a \right]_{\widetild{\rho}}
\end{equation}
le segment contragrédient de $\Delta$.

\label{lies}

\begin{defi}
Soient $\Delta=\left[a,b\right]_\rho$ et $\Delta'=\left[a',b'\right]_{\rho'}$
des segments. 
\begin{enumerate}
\item 
On dit que $\Delta$ {\it précède} $\Delta'$ si l'on peut extraire 
de la suite~:
\begin{equation*}
(\rho\nu_{\rho}^a ,\dots,\rho\nu_{\rho}^{b}, 
\rho'\nu_{\rho'}^{a'},\dots,\rho'\nu_{\rho'}^{b'})
\end{equation*}
une sous-suite qui soit un segment de longueur strictement 
supérieure à $n(\Delta)$ et $n(\Delta')$.
\item
On dit que $\Delta$ et $\Delta'$ sont \textit{liés} si $\Delta$ 
précède $\Delta'$ ou si $\Delta'$ précède $\Delta$. 
\end{enumerate}
\end{defi}

Remarquons que, si $\Delta$ et $\Delta'$ sont liés, les ensembles 
$\ZZ_\rho$ et $\ZZ_{\rho'}$ sont égaux. 

\subsection{}
\label{multisegment}

Deux segments $\left[a,b\right]_{\rho}$ et $\left[a',b'\right]_{\rho'}$ 
sont dits \textit{équivalents} s'ils ont la même longueur et si 
$\rho\nu_{\rho}^a$ est isomorphe à $\rho'\nu_{\rho'}^{a'}$. 
On note $\seg_{\CR}$ l'ensemble des classes d'équivalence de seg\-ments.  

\begin{defi}
Un \textit{multisegment} est un élément de $\NN(\seg_{\CR})$.
On note~:
\begin{equation*}
\MS=\MS_\CR
\end{equation*}
l'ensemble des multisegments. 
\end{defi}

Soit $\m= \Delta_1 + \dots + \Delta_r$ un multisegment.  
Les applications longueur, degré et support sont 
prolongées à $\MS$ par additivité, \ie qu'on note~: 
\begin{equation*}
n(\m)=\underset{1 \< i \< r}{\sum}n\( \Delta_i \),
\quad
\deg\(\m \)= \underset{1 \< i \< r}{\sum}\deg\( \Delta_i \),
\quad 
\supp(\m)=\underset{1 \< i \< r}{\sum}\supp\( \Delta_i \),
\end{equation*}
la longueur, le degré et le support cuspidal de $\m$ repectivement.  
\'Etant donné un support cuspidal $\sss\in\Dive\(\Cusp\)$, on note~: 
\begin{equation*}
\MS(\sss)
\end{equation*}
l'ensemble de tous les multisegments de support $\sss$.  

On note 
$\cCO{\m}$ l'en\-sem\-ble des élé\-ments de $\Cusp_\CR$ apparaissant dans 
$\supp(\m)$ avec multiplicité $\>1$.  
On dit que le support de $\m$ est \textit{connexe} s'il existe un segment $\Delta$
tel que $\cCO{\m}=\cCO{\Delta}$. 
Dans ce cas on dit que $\m$ est à support connexe. 

\section{Les représentations banales}
\label{sectionbanal}

Dans cette section, on définit la notion de représentation banale et on donne 
quelques premières propriétés et exemples de telles représentations.

\subsection{}
\label{ordon}

On commence par définir la notion de multisegment banal.  

\begin{defi}
Un multisegment $\m$ est dit \textit{banal} si, pour toute représentation 
irré\-duc\-tible cuspidale $\rho$, il existe un élément de $\ZZ_\rho$ 
n'apparaissant pas dans $\supp(\m)$. 
\end{defi}

Un segment $\left[a,b\right]_\rho$ est banal si et seulement s'il est 
de longueur strictement inférieure à $e(\rho)$.

\begin{rema}
Notons que la somme de deux multisegments banals n'est pas toujours banale.  
Par exemple, supposons que l'image de $q$ dans $\mult\CR$ soit d'ordre $2$.
On note $1$ le caractère trivial de $\GL_1(\F)$ et $\nu$ son unique 
caractère non ramifié d'ordre $2$.
Alors les segments $\sy{1}$ et $\sy{\nu}$ 
sont banals tandis que leur somme ne l'est pas.
\end{rema}

Si $\m$ est un multisegment banal,
il existe des segments $\Delta_1,\dots,\Delta_r$ tels que, 
pour $i<j$, le segment $\Delta_i$ ne précède pas $\Delta_j$
et tels que $\m=\Delta_1+\dots+\Delta_r$.
Une telle famille $(\Delta_1,\dots,\Delta_r)$ s'appelle une 
\textit{forme rangée} de $\m$. 
En général, un mutisegment non banal ne possède pas de 
forme rangée. 

\subsection{}
\label{ConnBana}

Soit $m\>1$ un entier et soit $\rho$ une représentation irréductible 
cuspidale de $\G_m$ telle que $e(\rho)\>2$.  
On fixe un entier $0\<t\<e(\rho)-1$. 

\begin{defi}
Soient $\Delta=\left[a,b \right]_\rho$ et $\Delta=\left[a',b'\right]_{\rho}$ 
deux segments vérifiant la condi\-tion $0\<a,b,a',b'\<t$.  
On écrit~: 
\begin{equation*}
\Delta>\Delta'
\end{equation*}
si $a > a'$, ou bien si $a=a'$ et $b>b'$.
\end{defi}

On écrit aussi $\Delta\>\Delta'$ si 
$\Delta > \Delta'$ ou $\Delta=\Delta'$. 
La relation $\>$ est une relation d'ordre dépendant des choix de $\rho$ et 
$t$. 

\subsection{}
\label{ConnBanaMult}

Soit $\sss\in\Dive(\Cusp)$ un support \textit{connexe et banal}.  
On va prolonger l'ordre sur les segments en un ordre total sur $\MS(\sss)$.  
D'abord, il existe une représentation irréductible cuspidale $\rho$ 
(unique à isomorphisme près) et un unique entier $0\<t\<e(\rho)-1$ tels que~: 
\begin{equation*}
\supp(\sss)=\sy{\rho }+\sy{\rho\nu_\rho^{}}+\dots+\sy{\rho\nu_\rho^t}. 
\end{equation*}
Comme $\sss$ est banal, chaque terme de cette somme apparaît avec multiplicité 
$1$, \ie qu'on a 
$\cCO{\sss}=\left\{\sy{\rho},\sy{\rho\nu_\rho^{}},\dots,\sy{\rho\nu_\rho^t}\right\}$. 

On appelle \textit{forme ordonnée} d'un multisegment $\m\in \MS(\sss)$ 
l'unique famille~:
\begin{equation*}
(\Delta_1,\dots,\Delta_r)
\end{equation*}
de segments tels que 
$\m=\Delta_1+\dots+\Delta_r$ et $\Delta_1\>\Delta_2\>\dots\>\Delta_r$. 
Si le support $\sss$ n'est pas banal ou n'est pas connexe, 
la forme ordonnée de $\m$ peut ne pas exister ou ne pas être unique. 

On étend maintenant la relation $\>$ aux multisegments de 
support $\sss$.  

\begin{defi}
Soient $\m$ et $\m'$ des multisegments de support $\sss$, 
et soient $(\Delta^{}_1,\dots,\Delta^{}_r)$ et $(\Delta'_1,\dots,\Delta'_{r'})$
leurs formes ordonnées respectives. 
On écrit~: 
\begin{equation*}
\m>\m'
\end{equation*}
si l'on est dans l'un des cas suivants~:
\begin{enumerate}
\item
Il existe $1\<i<{\rm min}(r,r')$
tel que $\Delta_1^{}=\Delta'_1,\dots,\Delta_i^{}=\Delta'_i$ et 
$\Delta_{i+1}^{}>\Delta'_{i+1}$.
\item
On a $r>r'$ et 
$\Delta_1^{}=\Delta'_1,\dots,\Delta_{r'}^{}=\Delta'_{r'}$. 
\end{enumerate}
\end{defi}

On écrit aussi $\m\>\m'$ si $\m>\m'$ ou $\m=\m'$. 
La relation $\>$ est une relation d'ordre sur $\MS(\sss)$. 

\subsection{}

Soit $\sss=\sy{\rho_1}+\dots+\sy{\rho_r}$ un élément de $\Dive(\Cusp)$.  
La propriété suivante, prouvée dans \cite[\textsection 8.2]{MS}, montre l'importance de la 
notion de représentation banale.  

\begin{prop}
\label{propbanales}
Supposons que $\sss$ soit banal et que $r\>2$.  
Alors la représentation~:
\begin{equation*}
\rho_1 \times \dots \times \rho_r
\end{equation*}
ne contient aucun sous-quotient irréductible cuspidal.
\end{prop}

Le résultat analogue dans le cas complexe est dû à \cite{BZ1}.

\begin{defi}
On dit qu'une représentation irréductible est \textit{banale} si son support 
cuspidal est banal.  
\end{defi}

\begin{exem}
\begin{enumerate}
\item
Une représentation irréductible cuspidale $\rho$ est banale si et seulement si 
$e(\rho)\>2$. 
\item 
Une représentation irréductible banale est cuspidale si 
et seulement si elle est super\-cuspidale 
(voir \cite[Remarque 7.16]{MS}). 
\item 
Si la caractéristique $\ell$ de $\CR$ est banale pour $\G_m$, 
\ie si $\ell$ est différente de $p$ et ne divise aucun des entiers 
$q^{di}-1$ pour $i\in\{1,\dots,m\}$, où $d$ désigne le degré réduit 
de $\D$ sur $\F$ et $q$ le cardinal du corps résiduel de $\F$, 
alors toute représentation irréductible de $\G_m$ est banale. 
Ceci est valable en particulier si $\R$ est de caractéristique nulle. 
\item
Inversement, on suppose que la caractéristique $\ell$ de $\CR$ est 
non nulle, et soit $m$ un entier tel que toute représentation irréductible 
de $\G_m$ soit banale. 
On note $e$ l'ordre de $q^d$ dans $\FF_\ell^\times$ et $1_{\mult\D}$ le 
caractère trivial de $\mult\D$. 
Si $e\<m$, alors la représentation $\St(1_{\mult\D},m)$ n'est pas banale 
(voir \cite[Proposition 7.8]{MS}). 
On en déduit que $e>m$, \ie que $\ell$ est banale pour $\G_m$. 
\end{enumerate}
\end{exem}

Le corollaire suivant est immédiat après la proposition \ref{propbanales}.

\begin{coro}
\label{egalitedesupports}
Soit $\pi$ une représentation irréductible banale.  
Alors le support cuspidal de $\pi$ est égal à son support supercuspidal.  
\end{coro}

\begin{rema}
La réciproque n'est pas vraie.  
Si l'image de $q$ dans $\mult\CR$ est d'ordre $2$, 
le caractère trivial de $\GL_2(\F)$ n'est pas banal
mais ses supports cuspidal et super\-cuspidal coïncident. 
\end{rema}

\subsection{}

Soit $\rho$ une représentation irréductible cuspidale de $\G_m$, 
et soit $\Delta= \left[a,b\right]_\rho$ un segment banal.
On pose~:
\begin{equation}
\Pi(\Delta)=\rho\nu_\rho^a\times\dots\times\rho\nu_\rho^b
\end{equation}
et on note $n$ la longueur de $\Delta$.
La proposition suivante associe à $\Delta$ deux représentations 
irréductibles de $\G_{mn}$.

\begin{prop} 
\label{segm}
\begin{enumerate}
\item 
L'induite $\Pi(\Delta)$ possède une unique sous-re\-pré\-sen\-ta\-tion 
ir\-ré\-ductible, notée $\Z(\Delta)$.  
C'est l'unique représentation irréductible de $\G_{mn}$, à
isomor\-phis\-me près, telle que~: 
\begin{equation*}
\rp_{(m,\dots,m)}\left(\Z(\Delta)\right)=\rho\nu_\rho^a\otimes\dots\otimes\rho\nu_\rho^b.
\end{equation*}
\item 
L'induite $\Pi(\Delta)$ possède un unique quotient ir\-ré\-ductible, noté 
$\L(\Delta)$.  
C'est l'unique représentation irréductible de $\G_{mn}$, à
isomorphisme près, telle que~: 
\begin{equation*}
\rp_{(m,\dots,m)}\left(\L(\Delta)\right)=\rho\nu_\rho^b\otimes\dots\otimes\rho\nu_\rho^a.
\end{equation*}
\end{enumerate}
\end{prop}

\begin{proof}
Puisque $\Delta$ est banal, tout élément de $\supp(\Delta)$ 
apparaît avec mul\-ti\-pli\-ci\-té $1$.  
D'après la proposition \ref{mmm1}, il existe une unique 
sous-re\-pré\-sentation ir\-ré\-duc\-tible et un unique quotient irréductible 
dans $\Pi(\Delta)$. 
D'après la proposition \ref{propbanales}, si le segment 
$\Delta$ est de longueur $\>2$, l'induite 
$\Pi(\Delta)$ ne contient pas de facteur cuspidal.  
La propriété sur le foncteur de Jacquet se prouve alors comme dans 
\cite[\textsection 2]{Ze2} ou \cite[Proposition 2.7]{Tadic}. 
Voir aussi \cite[Lemme 7.26]{MS}
\end{proof}

\begin{rema}
On a 
$\L\left(\Delta^{\vee}\right)\simeq \L \left( \Delta \right)^{\vee}$ et 
$\Z\left(\Delta^{\vee}\right)\simeq \Z \left( \Delta \right)^{\vee}$. 
\end{rema}

\section{Classification des représentations banales}
\label{sec}

Dans cette section on classifie toutes les représentations banales en termes 
de multisegments banals.  

\subsection{} 
\label{seccc} 

Soit $m\> 1$ un entier, soit $\rho$ une représentation irréductible 
cuspidale de $\G_m$ et soit $\Delta=\left[a,b \right]_\rho$ un segment.  
Pour montrer les théorèmes \ref{Z21} et \ref{Z22}, on notera indifféremment 
$\langle\Delta\rangle=\langle a,b\rangle_\rho$ la 
représentation $\Z(\left[a,b \right]_\rho)$ ou 
$\L(\left[-b,-a\right]_{{\rho}})$, comme dans \cite[\textsection 7.5.1]{MS}.

Dans le cas où $\langle\Delta\rangle$ désigne 
$\Z(\left[a,b \right]_\rho)$ (res\-pec\-ti\-ve\-ment
${\L(\left[-b,-a\right]_{{\rho}})}$) on note $\mu_\rho$ 
le carac\-tè\-re $\nu_\rho$ (respectivement $\nu_\rho^{-1}$). 

\medskip

La proposition suivante résume la proposition \ref{segm} 
et montre l'intérêt de cette notation.

\begin{prop}\label{segmbanal}
Soit $\left[a,b\right]_\rho$ un segment banal et soit $\pi$ une 
représentation ir\-ré\-duc\-ti\-ble. 
Les conditions suivantes sont équivalentes~: 
\begin{enumerate}
\item 
La représentation $\pi$ est isomorphe à $ \left<\Delta \right>$.
\item 
$\pi$ est l'unique sous-représentation irréductible de
$\rho \mu_{\rho}^a\times\rho\mu_{\rho}^{a+1}\times\dots\times\rho\mu_{\rho}^{b}$. 
\item 
On a 
$\rp_{(m,\dots,m)}\( \pi \)=\rho \mu_{\rho}^a\otimes\rho \mu_{\rho}^{a+1} 
\otimes \dots \otimes\rho \mu_{\rho}^{b}$. 
\end{enumerate}
\end{prop}

Le théorème suivant est un cas particulier de \cite[Théorème 7.38]{MS}. 
On remarquera que, dans le cas banal, la preuve donnée dans \cite{MS} 
se simplifie notablement. 

\begin{theo}
\label{nuevo2}
Soient $\Delta_1,\dots,\Delta_r$ des segments banals. 
Si, pour tous $1\<i,j\<r$, les segments $\Delta_i$ 
et $\Delta_j$ sont non liés, alors la représentation~:
\begin{equation*}
\left< \Delta_1\right>\times \dots \times \left< \Delta_r\right>
\end{equation*} 
est irréductible. 
\end{theo}

\subsection{}

Le théorème suivant englobe les théorèmes \ref{Z21} et \ref{Z22}.

\begin{theo}
\label{Z2}
\begin{enumerate}
\item 
Soit $(\Delta_1,\dots,\Delta_r)$ une forme rangée d'un 
multisegment banal $\m$. 
Alors 
$\left<\Delta_1 \right>\times \dots \times \left<\Delta_r \right>$ 
admet une unique sous-représentation irréductible, ne dépendant que 
de $\m$ et notée~:
\begin{equation*}
\langle\m\rangle.
\end{equation*}
Elle est banale et sa multiplicité dans 
$\sy{\left<\Delta_1 \right>\times \left<\Delta_2 \right>
\times\dots \times \left<\Delta_r\right>}$ est égale à $1$. 
\item 
Soient $\m$ et $\m'$ des multisegments banals. 
Alors $ \left<\m\right>$ 
et $\left<\m'\right>$ sont isomorphes si et seulement 
si $\m=\m'$.
\item 
Soit $n\>1$. 
Toute représentation irréductible banale de $\G_n$ est de la 
forme $\left<\m\right>$, où $\m$ est un multisegment banal de degré $n$. 
\end{enumerate}
\end{theo}

\begin{rema}
\label{rem1}
Si $(\Delta'_1,\dots,\Delta'_r)$ est une autre forme rangée de 
$\m$, alors~: 
$$\left<\Delta_1 \right>\times \dots \times \left<\Delta_r \right> \simeq 
\left<\Delta'_1 \right>\times \dots \times \left<\Delta'_r \right>$$ 
d'après le théorème \ref{nuevo2}.
\end{rema}

Le reste de cette section sera consacré à la preuve de ce théorème.

\subsection{}

Dans ce paragraphe on prouve la partie (1) du théorème \ref{Z2}. 
Soit $\m$ un multisegment banal non nul.  
Il existe  $u\>1$ et des familles de multisegments 
$(\Delta_1^{(i)},\dots,\Delta_{n_i}^{(i)})$, $1 \< i \<u$, telles que~:
\begin{equation}
\label{numero}
(\Delta_1^{(1)},\dots,\Delta_{n_{1}}^{(1)},
\Delta_{1}^{(2)},\dots,\Delta_{n_{2}}^{(2)},\dots, 
\Delta_1^{(u)},\dots,\Delta_{n_{u}}^{(u)})
\end{equation}
soit une forme rangée de $\m$ 
et que, pour tous $1 \< i \< u$ et $1\< j,k \< n_i$, on ait~: 
\begin{equation}\label{primera}
b\Big(\Delta_{j}^{(i)}\Big)= b\Big (\Delta_{k}^{(i)}\Big),
\end{equation}
c'est-à-dire que $\Delta_{j}^{(i)}$ et $\Delta_{k}^{(i)}$ ont la même
extrémité finale, que l'on désigne par $\rho_i$.
Puisque le multisegment $\m$ est banal, on peut même supposer 
que~:
\begin{equation}
\label{segunda}
\rho_i\notin\Delta^{(j)}_k
\end{equation}
pour tous $1 \< i < j \< u$ et $1\< k \< n_j$
(\ie que, pour $1\< l \< n_i$ et $1\< k \< n_j$, 
le segment $ \Delta^{(i)}_l$ ne précède pas $\Delta^{(j)}_k$). 
On note $\a$ la partition~: 
$$\a=\big( 
\deg\big ( \Delta_1^{(1)}+\dots+\Delta_{n_{1}}^{(1)}\big), 
\deg\big ( \Delta_1^{(2)}+\dots+\Delta_{n_{2}}^{(2)}\big),
\dots, 
\deg\big ( \Delta_1^{(u)}+\dots+\Delta_{n_{u}}^{(u)}\big) 
\big) $$ 
et on pose~:
$$\pi_0 = \overset{u}{\underset{i=1}{\bigotimes}} 
\big(\langle\Delta_1^{(i)}\rangle 
\times\dots\times 
\langle\Delta_{n_{i}}^{(i)}\rangle\big)$$ 
qui est une représentation irréductible (par le théorème \ref{nuevo2}) 
de $\M_\a$.

\begin{lemm}
\label{unosolo} 
La représentation $\pi_0$ apparaît avec multiplicité $1$ dans~: 
\begin{equation*}
\rp_\a\big(
\langle\Delta_1^{(1)}\rangle\times \dots \times\langle\Delta_{n_{1}}^{(1)}\rangle
\times \dots \times 
\langle\Delta_1^{(u)}\rangle\times \dots \times\langle\Delta_{n_{u}}^{(u)}\rangle\big).
\end{equation*}
\end{lemm}

\begin{proof}
Par nos hypothèses \eqref{primera}, \eqref{segunda} et le lemme géométrique 
(voir le paragraphe \ref{lemmegeo}), 
la représentations $\pi_0$ satisfait aux conditions de la proposition
\ref{mmm1}.  
Le lemme est donc une conséquence directe de cette proposition. 
\end{proof}

\begin{prop}
\label{unosolos}
Avec les notations du lemme \ref{unosolo}, on a~:
\begin{enumerate}
\item 
La représentation~: 
$$\big\langle\Delta_1^{(1)}\big\rangle\times \dots
\times \big\langle\Delta_{n_{1}}^{(1)} \big\rangle\times \dots \times
\big\langle\Delta_1^{(u)}\big\rangle\times \dots \times
\big\langle\Delta_{n_{u}}^{(u)}\big\rangle$$ 
possède une unique sous-représentation
irréductible~; sa multiplicité dans l'induite est $1$. 
\item 
La représentation~: 
$$\big\langle\Delta_{n_u}^{(u)}\big\rangle\times \dots
\times \big\langle\Delta_{{1}}^{(u)} \big\rangle\times \dots \times
\big\langle\Delta_{n_1}^{(1)}\big\rangle\times \dots \times
\big\langle\Delta_{1}^{(1)}\big\rangle$$ 
possède un unique quotient irréductible
irréductible~; sa multiplicité dans l'induite est $1$. 
\end{enumerate}
\end{prop}

\begin{proof}
Ceci se déduit du lemme \ref{unosolo} et de la proposition \ref{mmm1}. 
\end{proof}

Nous prouvons maintenant le point (1) du théorème \ref{Z2}. 
Soit $(\Delta_1,\dots,\Delta_r)$ une forme rangée d'un 
multisegment banal $\m$. 
La famille \eqref{numero} est une autre forme rangée de $\m$. 
Le résultat est une conséquence de la remarque \ref{rem1} et de 
la proposition \ref{unosolos}. 

\begin{rema}
\label{Z12}
On déduit de même, par la proposition \ref{unosolos} (2), que la 
représentation induite 
$\left<\Delta_r\right>\times \dots \times \left<\Delta_1 \right>$ 
possède un unique quotient irréductible.  
Le module de Jacquet de ce quotient contient la représentation $\pi_0$ 
du lemme \ref{unosolo} et il apparaît avec mul\-ti\-pli\-ci\-té $1$ dans 
l'induite. 
Puisque, par le lemme \ref{unosolo}, la représentation 
$\left<\m\right> $ est le seul sous-quotient de cette induite 
dont le module de Jacquet contienne $\pi_0$, on déduit que $\left<\m\right>$
\textit{est aussi le seul quotient irréductible de} 
$ \left<\Delta_r\right>\times \dots \times \left<\Delta_1\right>$. 
\end{rema}

\subsection{} 

Si $\m=\Delta_1+\dots+\Delta_r$ est un multisegment banal,
on note $\widetild{\m}$ le multisegment~: 
$$\widetild{\m}=\widetild{\Delta_1}+\dots+\widetild{\Delta_r}.$$ 
On a le résultat suivant. 

\begin{prop} 
La représentation $\left<\widetild{\m} \right>$ est isomorphe à 
$\left<\m\right>^\vee$. 
\end{prop}

\begin{proof}
On peut supposer que la famille $(\Delta_1,\dots,\Delta_r)$ est une forme rangée 
de $\m$. 
D'après la remarque \ref{Z12},  la représentation $\left<\m\right>$
est la seule sous-re\-pré\-sen\-ta\-tion irréductible de la représentation induite 
$\left<\Delta_1 \right>\times \dots \times \left<\Delta_r \right>$ et le seul 
quotient irréductible de 
$\left<\Delta_r \right>\times \dots \times\left<\Delta_1 \right>$. 
Par passage à la contragrédiente, 
on en déduit que $\widetild{ \left<\m \right>}$ est la seule sous-représentation 
irréductible de $\left<\widetild{\Delta}_r \right>\times \dots \times 
\left<\widetild{\Delta}_1 \right>$ et le seul quotient ir\-ré\-duc\-ti\-ble de 
$\left<\widetild{\Delta}_1 \right>\times \dots \times 
\left<\widetild{\Delta}_r \right> $. 
Comme $(\widetild{\Delta}_r,\dots,\widetild{\Delta}_1)$ est une forme rangée 
de $\m^\vee$, on a le résultat. 
\end{proof}

\begin{prop}
\label{despegados2}
Soient $\m_1, \dots, \m_t$ des multisegments banals tels que, si $i \neq j$,
aucun segment de $\m_i$ ne soit lié à un segment de $\m_j$.  
Alors~: 
\begin{equation*}
\left< \m_1\right> \times \dots \times \left< \m_t \right>\simeq 
\left< \m_1 + \dots+ \m_t \right>.
\end{equation*}
\end{prop}

\begin{proof}
Par récurrence, on se ramène au cas où $t=2$.
Soient $(\Delta_1,\dots,\Delta_r)$ et 
$(\Delta'_1,\dots,\Delta'_{r'})$ des formes rangées 
de $\m_1$ et $\m_2$ respectivement.
Puisqu'aucun segment de $\m_1$ n'est lié à un segment de
$\m_2$, la famille $(\Delta^{}_1+\dots+\Delta^{}_r+\Delta'_1+\dots+\Delta'_{r'})$ 
est une forme rangée du multisegment \textit{banal} $\m_1+\m_2$. 
Ainsi la représentation~: 
\begin{equation}
\label{inv1}
\left<\Delta^{}_1\right>\times \dots \times \left<\Delta^{}_r\right> \times
\left<\Delta'_1\right> \times \dots \times \left<\Delta'_{r'} \right>
\end{equation}
a, par le théorème \ref{Z2}(1) une unique sous-représentation irréductible et elle
est isomorphe à la représentation $\left<\m_1+\m_2\right>$.  
Puisque $\left< \m_1\right> \times \left< \m_2 \right>$ est une 
sous-représentation de \eqref{inv1}, on en déduit que 
$\left<\m_1+\m_2\right>$ est l'unique sous-représentation irréductible de 
$\left<\m_1\right> \times \left< \m_2 \right>$. 
De même, d'après la remarque \ref{Z12}, la représentation~: 
\begin{equation}
\label{inv2}
\left<\Delta^{}_r\right>\times \dots \times\left<\Delta^{}_1\right> 
\times \left<\Delta'_{r'}\right> \times \dots \times\left<\Delta'_1 \right> 
\end{equation}
a un unique quotient irréductible isomorphe à $\left<\m_1+\m_2\right>$.
Comme $\left< \m_1\right> \times \left< \m_2 \right>$ est un quotient de
\eqref{inv2}, on trouve que 
$\left<\m_1+\m_2\right>$ est l'unique quotient irréductible de 
$\left<\m_1\right> \times \left< \m_2 \right>$. 
Or, d'après le théorème \ref{Z2}, la représentation 
$\left<\m_1+\m_2\right>$ apparaît avec
multiplicité $1$ dans \eqref{inv1}. 
On en déduit que $\left< \m_1\right> \times \left< \m_2 \right>$ est 
isomorphe à $\left<\m_1+\m_2\right>$.
\end{proof}

Ce résultat permet de nous ramener au cas des multisegments
banals de support fixé et connexe. 

\subsection{}
\label{connexe} 

Comme au paragraphe \ref{ConnBana}, 
soit $\sss \in \Dive(\Cusp)$ connexe et banal et écrivons~:
\begin{equation*}
\supp(\sss)=\sy{\rho}+\sy{\rho\mu_\rho}+\dots+\sy{\rho\mu_\rho^t}
\end{equation*}
où $\rho$ est une représentation irréductible cuspidale de $\G_m$
et $0\<t\<e(\rho)-1$. 
Soit un multisegment $\m\in\MS(\sss)$ 
et soit $(\Delta_1,\dots,\Delta_r)$ sa forme ordonnée. 
Pour chaque $i$, on écrit $\Delta_i= \left[a_i,b_i\right]_{\rho}$.
Alors, par le théorème \ref{Z2} (1), on peut déduire facilement 
(par réciprocité de Frobenius) que~: 
\begin{equation}
\label{pequena}
\rho \mu_\rho^{a_1}\otimes\dots\otimes\rho\mu_\rho^{b_1}\otimes
\rho\mu_\rho^{a_2}\otimes\dots\otimes\rho\mu_\rho^{b_2}\otimes 
\dots\otimes\rho\mu_\rho^{a_r}\otimes\dots\otimes\rho\mu_\rho^{b_r} 
\end{equation} 
apparaît dans le module de Jacquet
$\rp_{(m,\dots,m)}( \left<\Delta_1, \dots, \Delta_r\right>)$.

Si $\m'=\Delta'_1+\dots+\Delta'_{r'}$ est un multisegment de support $\sss$ 
tel que $\m'<\m$, alors, d'après le lemme géométrique, la représentation 
\eqref{pequena} n'apparaît pas dans le module de Jacquet 
$\rp_{(m,\dots,m)}(\left<\Delta'_1\right> \times \dots \times 
\left<\Delta'_{r'}\right>)$.
Par le théorème \ref{Z2}(1) et l'exactitude du foncteur de
Jacquet, elle n'apparaît pas dans 
$\rp_{(m,\dots,m)}( \left<\Delta'_1, \dots, \Delta'_{r'}\right>)$, 
la représentation $\left<\Delta'_1,\dots, \Delta'_r\right>$ étant 
une sous-représentation de $ \left<\Delta'_1\right>
\times \dots \times \left< \Delta'_{r'}\right>$. 
On en déduit que, si $\m,\m' \in \MS(\sss)$ et $\m \neq \m'$, 
alors $\left< \m \right>\not\simeq \left<\m' \right>$. 

\subsection{}
\label{preuvethe} 

On prouve dans ce paragraphe la partie (2) du théorème \ref{Z2}.
Soient $\m$ et $\m'$ deux multi\-segments banals. 
On décompose $\m$ sous la forme~:
\begin{equation*}
\m=\m_1+\m_2+\dots+\m_t
\end{equation*}
 avec $\m_i$ à support connexe pour tout $1 \< i \< t$ et 
$\m_i,\m_j$ à supports disjoints si $i\neq j$. 
De même, on écrit $\m'=\m'_1+\m'_2+\dots+\m'_{t'}$.
Supposons que $\m \neq \m'$ et montrons que les représentations 
$\left< \m \right>$ et $\left<\m'\right>$ ne sont pas isomorphes. 
D'après la proposition \ref{despegados2}, on peut supposer que
$\cC{\m_1^{}}= \cC{\m'_1}$ et que $\m^{}_1 > \m'_1$.  
Notons $\m_1=\Delta_1+\dots+\Delta_r$ avec 
$\Delta_i= \left[a_i,b_i\right]_{\rho}$ pour tout $i$, 
où $\rho$ est une représentation irré\-duc\-tible cuspidale de $\G_m$. 
De façon analogue, notons $\m'=\Delta'_1+\dots+\Delta'_{r'}$. 

D'après la partie (1) du théorème \ref{Z2}, il existe une représentation $\pi$
de la forme \eqref{pequena} et une repré\-sen\-ta\-tion $\pi' \in \G_{n'}$, 
avec $n'=\deg\(\m\)- \deg(\m_1)$, telles que la représentation 
$\pi \otimes \pi'$ 
apparaisse dans $\rp_{(m,m,\dots,m,n')}( \left<\m\right>)$.  
Comme dans le paragraphe \ref{connexe}, par 
le lemme géo\-métrique et l'hypothèse $\m^{}_1>\m'_1$, pour toute
représentation irréductible $\pi' \in \G_{n'}$, la représentation 
$\pi\otimes \pi'$ n'apparaît pas dans 
$\rp_{(m,m,\dots,m,n')}( \left<\Delta'_1\right> \times \dots \times \left<
 \Delta'_{r'}\right>)$.
D'après la partie (1) du théorème \ref{Z2} et par exac\-titude du foncteur de 
Jacquet, elle n'apparaît pas dans 
$\rp_{(m, m,\dots,m,n')}( \left<\m\right>)$. 
On en déduit que $\left< \m \right> \not\simeq \left< \m' \right>$.

\subsection{} 

Montrons maintenant la partie (3) du théorème \ref{Z2}. 
Soit $\pi$ un représentation banale et soit $\sss$ son support cuspidal.  
On note $\Irr(\sss)$ l'ensemble des représentations irréductibles de 
support cuspidal $\sss$ et $\MS(\sss)$ l'ensemble des multisegments de 
support $\sss$.  
Les ensembles $\Irr(\sss)$ et $\MS(\sss)$ sont finis et, d'après le théorème 
\ref{Z2} (2) qu'on vient de montrer, l'application $\m \mapsto \left< \m \right>$ 
est une injection de $\MS(\sss)$ vers $\Irr(\sss)$.  
On va prouver ici qu'elle est bijective. 

Pour cela, il suffit de montrer que les ensembles $\Irr(\sss)$ et 
$\MS(\sss)$ ont le même cardinal.  
On va prouver que l'application $\m \mapsto \Z(\m)$ est surjective. 
La preuve est la même que celle de \cite[6.7]{Ze2} dans le cas 
$\D=\F$ et $\CR=\CC$. 
Il faut d'abord étudier le cas de deux segments. 

\begin{lemm}
\label{2seglies}
Soit $\rho$ un représentation irréductible cuspidale de $\G_m$, 
et soient deux segments banals liés 
$ \left[ a,b \right]_\rho$ et $ \left[ a',b' \right]_{\rho}$ 
tels que le multisegment 
$ \left[ a,b \right]_\rho+\left[ a',b' \right]_{\rho} $ soit banal. 
\begin{enumerate}
\item
La représentation~:
\begin{equation}
\label{ProdDeuxSegBan}
\Z( \left[a,b \right]_\rho ) \times\Z( \left[a',b' \right]_\rho)
\end{equation}
est indécomposable de longueur $2$. 
\item 
On suppose que $\left[ a,b \right]_\rho$ précède $\left[ a',b' \right]_{\rho}$.
Alors l'unique sous-représentation irré\-duc\-ti\-ble de \eqref{ProdDeuxSegBan} 
est $\Z( \left[a,b' \right]_\rho) \times\Z( \left[a',b\right]_\rho)$ et son 
unique quotient irréductible est une sous-re\-présenta\-tion de 
$ \Z( \left[a',b' \right]_\rho) \times \Z( \left[a,b \right]_\rho)$.  
\end{enumerate}
\end{lemm}

\begin{proof}
Le résultat est vrai dans le cas déployé, 
\ie lorsque $\D$ est égale à $\F$~: l'argument de Zelevinski 
\cite[Proposition 4.6]{Ze2} 
dans le cas complexe est encore valable. 
En effet, dans la preuve, on utilise trois outils qui sont valables 
dans notre cadre~: 
\begin{enumerate}
\item[$\bullet$] 
la théorie des dérivées~;
\item[$\bullet$] 
le fait que dans l'induite $\Z( \left[a,b \right]_\rho) 
\times\Z( \left[a',b' \right]_\rho)$ tous les facteurs ont le même 
support cuspidal (proposition \ref{egalitedesupports})~;
\item[$\bullet$] 
le cas où le supports de $ \left[ a,b \right]_\rho$ et $ \left[ a',b' 
\right]_{\rho}$ sont disjoints (voir \cite[\textsection 2]{Ze2}), 
qui s'étend ici comme on l'a vu dans le théorème \ref{nuevo2}. 
\end{enumerate}
Pour prouver le résultat dans le cas général, on applique la méthode 
du changement de groupe exposée dans \cite[\textsection 5.4]{MS}
(voir notamment \textit{ibid.}, corollaire 5.34).
\end{proof}
Pour les représentations à la Langlands, voir la remarque \ref{mover2rema}.
\begin{rema}
Si le multisegment $ \left[ a,b \right]_\rho+ \left[ a',b' \right]_{\rho} $ 
n'est pas banal, alors l'induite~:
\begin{equation*}
\Z( \left[a,b \right]_\rho ) 
\times \Z( \left[a',b' \right]_\rho )
\end{equation*}
peut être de longueur strictement supérieure à $2$. 
\end{rema}

Soit $\pi$ une représentation irréductible banale de $\G_m$,
et soit $\Phi(\pi)$ l'ensemble des familles de segments~:
\begin{equation}
\label{PhiF}
\overrightarrow{\m} =\(\Delta_1, \dots, \Delta_r\)
\end{equation}
tels que $\pi$ soit une sous-représentation de 
$\Z\left(\Delta_1\right) \times \dots \times
\Z\left( \Delta_r\right)$.  
Cet ensemble est non vide et fini.  
Appelons \textit{inversion} de \eqref{PhiF} 
un couple d'indices $(i,j)$ tels que $i<j$ et $\Delta_i$
précède $\Delta_j$. 
On va prouver qu'il existe un élément
$\overrightarrow{\m} \in\Phi(\pi)$ sans inversion. 
Soit $\overrightarrow{\m}$ dans $\Phi(\pi)$ avec un nombre 
d'inversions minimal, qu'on écrit sous la forme \eqref{PhiF}.
Pour chaque $i$, on pose $\Delta_i=[a_i,b_i]_\rho$.  
Supposons que $\overrightarrow{\m}$ ait une inversion. 
Par le corollaire
\ref{mover2} (1), on peut bien supposer que cette inversion est de la forme
$(i,i+1)$, c'est-à-dire que $\Delta_i$ précède $\Delta_{i+1}$. 
D'après la proposition \ref{2seglies}, l'induite $\Z\left(\Delta_i\right)
\times \Z\left( \Delta_{i+1}\right)$ est composée de la représentation
$\Z([a_{i+1},b_i]_\rho) \times \Z([a_i,b_{i+1}]_\rho)$ 
et d'une sous-représentation irréductible de
$\Z\left(\Delta_{i+1}\right) \times \Z\left( \Delta_{i}\right)$. 
Ainsi~: 
\begin{enumerate}
\item 
ou bien $\pi$ est une sous-représentation de 
$ \Z\left(\Delta_1\right)\times \dots \times 
\Z\left(\Delta_{i+1} \right) \times \Z\left(
\Delta_{i}\right) \times \dots \times \Z\left( \Delta_r\right)$, 
\item 
ou bien $\pi$ est un sous-représentation de~:
\begin{equation*}
\Z\left(\Delta_1\right)
 \times \dots \times \Z\left([a_{i+1},b_i]_\rho\right) \times
 \Z\left([a_i,b_{i+1}]_\rho\right)\times \dots \times \Z\left(
 \Delta_r\right)
\end{equation*} 
\end{enumerate}
et les deux familles 
$\left(\Delta_1, \dots ,\Delta_{i+1},\Delta_{i}, \dots , \Delta_r\right)$ et 
$\left(\Delta_1,\dots,[a_{i+1},b_i]_\rho,[a_i,b_{i+1}]_\rho,\dots ,
 \Delta_r\right)$ ont strictement moins d'inversions
que $\overrightarrow{\m}$ d'après \cite[Lemma 6.7]{Ze2}. 

Cela met fin à la preuve du théorème \ref{Z2} et donc aussi aux théorèmes \ref{Z21} et \ref{Z22}.

\begin{rema}
\label{mover2rema}
Si $ \left[ a,b \right]_\rho$ et $ \left[ a',b' 
\right]_{\rho}$ sont deux 
segments banals liés et tels que le multisegment 
$ \left[ a,b \right]_\rho+\left[ a',b' \right]_{\rho} $ soit également 
banal, alors $ \L( \left[a,b \right]_\rho ) 
\times \L( \left[a',b' \right]_\rho )$ est aussi de longueur $2$ 
indécomposable. 
Par exemple, la preuve de \cite[Proposition 4.3]{Tadic} est 
valable ici en remplaçant le théorème du quotient de Langlands par le théorème 
\ref{Z22}. 
\end{rema}

\begin{coro}
\label{mover2}
Soit $\Delta_1+\dots+\Delta_r$ un multisegment banal. 
Les conditions suivantes sont équivalentes~: 
\begin{enumerate}
\item Pour tous $1 \< i,j \< r$, les segments $\Delta_i$ et $\Delta_j$ ne
 sont pas liés. 
\item La représentation $\left< \Delta_1 \right> \times \dots \times \left<
 \Delta_r \right>$ est irréductible. 
\end{enumerate}
\end{coro}

\begin{proof}
Ce corollaire 
découle du théorème \ref{nuevo2}, du lemme \ref{2seglies} et de la
remarque \ref{mover2rema}.
\end{proof}

\subsection{}
\label{complexe}

On suppose dans ce paragraphe que $\CR$ est le corps des nombres complexes.  
Si $\D$ est égale à $\F$, les théorèmes \ref{Z21} et \ref{Z22} sont prouvés dans 
\cite[Theorem 6.1]{Ze2} et \cite[Théorème 3]{Rodier} respectivement.  
Dans le cas où $\D \neq \F$, le 
théorème \ref{Z22} est prouvé dans \cite[\textsection 2]{Tadic} lorsque la 
caractéristique de $\F$ est nulle et dans \cite{BHLS} lorsque la 
caractéristique de $\F$ est positive. Notre preuve est purement locale et ne 
s'appuie pas sur \cite{DKV}. Le théorème de classification \ref{Z21} est aussi 
nouveau dans ce cas.  
Les résultats qui précèdent fournissent aussi, \textit{a posteriori}, une
nouvelle preuve du résultat suivant, dû à 
\cite[B.2.d]{DKV} et \cite[Théorème 1.1]{Bad}. 

\begin{theo}
Si $\s$ et $\s'$ sont deux représentations irréductibles 
de carré inté\-gra\-ble, alors l'induite
$\s\times\s'$ est irréductible. 
\end{theo}

\begin{proof}
Comme dans le cas où $\D$ est égale à $\F$ (la preuve non publiée de Bern\-stein 
s'étend au cas où $\D\neq\F$, voir \cite{Auel} ou \cite{BHLS} pour plus de 
détails), il existe des segments centrés $\Delta$ 
et $\Delta'$ tels que $\s=\L(\Delta)$ et $\s'=\L(\Delta')$.  
Étant centrés, $\Delta$ et $\Delta'$ ne sont pas liés et le résultat découle 
maintenant du théorème \ref{nuevo2}. 
\end{proof}

\section{Relèvement d'une représentation banale} \label{rele}

Soit $\ell$ un nombre premier différent de $p$. 
On note $\qlb$ une clôture algébrique du corps $\mathbf{Q}_\ell$ 
des nombres $\ell$-adiques, $\zlb$ son anneau d'entiers 
et $\flb$ son corps résiduel. 
Dans cette section, on montre que toute $\flb$-représentation irréductible 
banale admet un relèvement à $\qlb$.  

\subsection{}
\label{DefRepEnt}

Soit un entier $m\>1$.
Une représentation de $\G_m$ sur un $\qlb$-espace vectoriel $\V$ est dite 
{\it entière} si elle est ad\-mis\-si\-ble et si elle admet une 
{\it structure entière}, \ie un sous-$\zlb$-module de $\V$ sta\-ble par $\G_m$ et
engendré par une base de $\V$ sur $\qlb$. 

Une $\qlb$-représentation cuspidale est entière si, et seulement si, son 
caractère central est à valeurs dans $\zlb$. Une $\qlb$-représentation 
irréductible est entière si, et seulement si, son support cuspidal est formé 
de $\qlb$-représentations cuspidales entières. 

Si $\pi$ est une représentation irréductible entière de $\G_m$ sur un 
$\qlb$-espace vec\-to\-riel $\V$, alors, d'après \cite[Theorem 1]{VigW}, 
toutes ses structures en\-tières sont de type fini comme $\zlb\G$-module.  
Si $\mathfrak{v}$ est une structure entière de $\pi$, 
la re\-pré\-sen\-ta\-tion de $\G_m$ sur le $\flb$-espace vectoriel 
$\mathfrak{v}\otimes\flb$ est de lon\-gueur finie et sa
semi-sim\-pli\-fi\-ca\-tion ne dépend pas du choix de la struc\-tu\-re
entière d'après \cite[II.5.11]{Vig1}.
On note $\r_{\ell}(\pi)$ cette semi-simplification, qu'on ap\-pel\-le 
{\it réduction} de $\pi$ et qui ne dépend 
que de sa classe d'iso\-mor\-phisme $\sy{\pi}$.
Par li\-néa\-rité, on en déduit un homo\-mor\-phis\-me de groupes~: 
\begin{equation}
\label{HomRes}
\r_{\ell}:\Gg_{\qlb}^{{\rm ent}}(\G_m)\to\Gg_{\flb}(\G_m),
\end{equation}
où $\Gg_{\qlb}^{{\rm ent}}(\G_m)$ désigne le sous-groupe de $\Gg_{\qlb}(\G_m)$
engendré par l'ensemble des clas\-ses d'iso\-mor\-phis\-me de 
$\qlb$-représen\-ta\-tions irréductibles entières de $\G_m$. 

On appelle \textit{relèvement} d'une $\flb$-représentation irréductible $\pi$ 
de $\G_m$ une $\qlb$-représen\-tation entière $\tilde\pi$ de $\G_m$ telle que 
$\sy{\pi} = \r_\ell(\tilde\pi)$. 
Si un tel relèvement existe, on dit que $\pi$ se relève. 

\subsection{} 

Soit un entier $m\>1$.
Le but de cette section est de montrer le théorème suivant~:

\begin{theo}
\label{releverbanales}
Soit $\pi$ une $\flb$-représentation irréductible \textit{banale} de 
$\G_m$. 
Alors $\pi$ admet un relèvement. 
\end{theo}

Remarquons que, pour les représentations cuspidales banales, le théorème 
\ref{releverbanales} est donné par la conjonction 
de \cite[Théorème 7.14]{MS}, qui implique qu'une représentation irréductible 
cus\-pi\-dale non supercuspidale n'est pas banale, et de \cite[Théorème 4.24]{MS}.  

\subsubsection{}

On étend d'abord les définitions des opérateurs d'entrelacement de
\cite[\textsection{5}]{Tadic} au cas des représentations banales.  
On reprend les notations du paragraphe \ref{seccc}.
Soient $\m$ un multisegment banal et 
$\( \Delta_1,\dots,\Delta_r \)$ une forme rangée de $\m$. 
On note 
$\I(\Delta_1, \dots , \Delta_r)$ la re\-pré\-sentation induite 
$\left< \Delta_1\right> \times\dots\times \left< \Delta_r \right>$. 

\begin{defi}
\label{opered}
On note~: 
\begin{equation*}
\Jj_\m: \I(\Delta_r, \dots , \Delta_1)
\to \I(\Delta_1, \dots , \Delta_r)
\end{equation*}
l'opérateur défini comme la composée~: 
\begin{equation*}
\I(\Delta_r, \dots , \Delta_1)
\iso{\a} \left< \Delta_1,\dots,\Delta_r \right> 
\iso{\b} \I(\Delta_1, \dots , \Delta_r)
\end{equation*}
où $\a$ et $\b$ sont respectivement la projection et l'inclusion définies par 
la remarque \ref{Z12} et le théorème \ref{Z2}.
\end{defi}

\begin{rema}
L'opérateur $\Jj_\m$ est un isomorphisme si et seulement si, 
pour tous entiers $1\<i,j\<r$, les segments $\Delta_i$ et $\Delta_j$ ne sont 
pas liés.  
\end{rema}

\begin{defi}
\label{opered2}
On note~: 
\begin{equation*}
\Jj'_\m: \I(\Delta_r, \dots , \Delta_1)
\to \I(\Delta_1, \dots , \Delta_r)
\end{equation*} 
l'opérateur défini comme la composée~:
\begin{equation*}
\label{ecua1} 
\begin{CD}
\I(\Delta_r,\dots,\Delta_3,\Delta_2,\Delta_1) 
@>{{\rm id}\times\dots\times{\rm id}\times\Jj_{\(\Delta_1,\Delta_2\)}}>>
& \I(\Delta_r,\dots,\Delta_3,\Delta_1,\Delta_2) \\ 
{} @>{{\rm id}\times\dots\times\Jj_{\(\Delta_1,\Delta_3\)}\times{\rm id}}>>
& \I(\Delta_r,\dots,\Delta_1,\Delta_3,\Delta_2) \\ 
& \vdots\\
{} @>{\Jj_{\(\Delta_1,\Delta_r\)}\times{\rm id}\times\dots\times{\rm id}}>>
& \I(\Delta_1,\Delta_r,\dots,\Delta_3,\Delta_2) \\
{} @>{{\rm id}\times\dots\times{\rm id}\times\Jj_{\(\Delta_2,\Delta_3\)}}>>
& \I(\Delta_1,\Delta_r,\dots,\Delta_2,\Delta_3) \\
& \vdots\\
{} @>{{\ }{\rm id}\times\dots\times\Jj_{\(\Delta_{r-1},\Delta_r\)}}>>
& \I(\Delta_1,\Delta_2,\Delta_3,\dots,\Delta_r)
\end{CD}
\end{equation*}
où, pour tous segments $\Delta,\Delta'$, l'opérateur 
$\Jj_{\(\Delta,\Delta'\)}$ est défini par la définition \ref{opered}. 
\end{defi}

\begin{lemm}\label{Tadmult}
Pour tout multisegment banal $\m$, on a $\Jj'_\m=\Jj_\m^{}$. 
\end{lemm}

\begin{proof}
La preuve de \cite[Lemma 5.1]{Tadic} est encore valable ici. 
\end{proof}

\subsubsection{} 

Soit $\rho$ une $\flb$-représentation irréductible 
cuspidale de $\G_m$ et soit $\tilde\rho$
un relèvement de $\rho$.  
Pour tout segment banal $\Delta=\left[a,b\right]_{\rho}$, 
on pose $\tilde\Delta=\left[a,b \right]_{\tilde\rho}$.  

\begin{prop}
\label{red1seg} 
La $\qlb$-représentation $\langle\tilde\Delta\rangle$ est entière et on a~: 
\begin{equation*}
\r_{\ell}(\langle\tilde\Delta\rangle)= \left< \Delta \right> .
\end{equation*}
\end{prop}
\begin{proof}
La représentation $\langle\tilde\Delta\rangle$ est entière d'après le
paragraphe \ref{DefRepEnt} et la proposition \ref{segm}.  
Comme $\Delta$ est banal, $\r_{\ell}(\langle\tilde\Delta\rangle)$ ne
contient pas de représentation de support cuspidal différent du support de
$\Delta$ d'après la proposition \ref{propbanales}.
Puisque la réduction modulo $\ell$ commute au foncteur de Jacquet, 
$\r_{\ell}(\langle\tilde\Delta\rangle)$ est une $\flb$-représentation 
dont le module de Jacquet relativement à la partition $(m,\ldots,m)$ 
est isomorphe à~: 
$$\mu_{\rho}^a \rho  \otimes \mu_{\rho}^{a+1} \rho \otimes
\dots \otimes \mu_{\rho}^{b} \rho.$$ 
La proposition découle alors de la proposition \ref{segmbanal}.
\end{proof}

Dans le cas général (\ie pour un segment qui n'est pas nécessairement banal), 
voir \cite[\textsection 9.7]{MS}.

\begin{prop} 
\label{red2seg}
Soient $\Delta_1=\left[a,b\right]_{\rho}$ et 
$\Delta_2=\left[a',b'\right]_{\rho}$ deux 
segments banals liés avec $1\< a+1\<a'\<b+1\< b'\<e(\rho)-1$. 
Alors $\langle\tilde\Delta_1 , \tilde\Delta_2 \rangle $ est
entière et on a~: 
$$\r_{\ell}( \langle\tilde\Delta_1,\tilde\Delta_2 \rangle)=\langle \Delta_1,\Delta_2 \rangle .$$ 
\end{prop}

Remarquons que la condition $1\< a+1\<a'\<b+1\< b'\<e(\rho)-1$
équivaut à dire que $\Delta_1$ précède $\Delta_2$ \textit{et que} 
$\tilde\Delta_1$ précède $\tilde\Delta_2$. 

\begin{proof}
La $\qlb$-représentation $ \langle\tilde\Delta_1 ,\tilde\Delta_2 \rangle $ est 
entière d'après le paragraphe \ref{DefRepEnt}.  
Comme le foncteur d'induction parabolique commute à la réduction, 
$\r_{\ell}( \langle\tilde\Delta_1 \rangle \times\langle\tilde\Delta_2 \rangle)$ 
est isomorphe à la semi-simplifiée de
$ \left< \Delta_1 \right> \times \left< \Delta_2 \right> $, 
c'est-à-dire, par la proposition \ref{2seglies}, à~: 
$$ \left< \Delta_1 , \Delta_2 \right> \oplus \( \left<a,b'\right>_\rho 
\times \left<a',b\right>_\rho \).$$ 
Puisque, d'un autre côté, par la proposition \ref{2seglies}, la représentation 
$ \langle\tilde\Delta_1 \rangle \times \langle\tilde\Delta_2\rangle $ est 
composée des représentations 
$ \langle\tilde\Delta_1 ,\tilde\Delta_2 \rangle$ et 
$\left<a,b'\right>_{\tilde\rho} \times \left<a',b\right>_{\tilde\rho}$, 
on a~: 
$$\r_{\ell}( \langle \tilde\Delta_1 \rangle \times \langle \tilde\Delta_2 \rangle)=
\r_{\ell}(  \langle  \tilde\Delta_1 ,\tilde\Delta_2 \rangle )+
\r_{\ell}( \left<a,b'\right>_{\tilde\rho} \times \left<a',b\right>_{\tilde\rho})$$
et, par la proposition \ref{red1seg} et le théorème \ref{nuevo2}~: 
$$\r_{\ell}(\left<a,b'\right>_{\tilde\rho} \times 
\left<a',b\right>_{\tilde\rho}) 
= \left<a,b'\right>_{\rho} \times \left<a',b\right>_{\rho}$$ 
et donc on trouve~: 
$$\r_{\ell}( \langle \tilde\Delta_1,\tilde\Delta_2 \rangle)= \left< \Delta_1,\Delta_2 \right>,$$
ce qui termine la démonstration.
\end{proof}

\subsubsection{} 

Voyons maintenant que l'opérateur $\Jj_\m$ défini à la définition 
\ref{opered} se réduit bien modulo $\ell$. 

\begin{prop}
\label{enti} 
Soient
 $\Delta_1=\left[a,b\right]_{\rho}$ et 
$\Delta_2=\left[a',b'\right]_{\rho}$ deux segments banals liés avec 
$1\< a+1\<a'\<b+1\< b'\<e(\rho)-1$.  
Notons $\ll_1$ et $\ll_2$ deux structures entières dans 
$ \langle \tilde\Delta_1 \rangle $ et $ \langle \tilde\Delta_2 \rangle $ 
respectivement. 
Alors il existe $\ll'_1$ et $\ll'_2$ deux structures entières dans 
$ \langle  \tilde\Delta_1  \rangle $ et $ \langle  \tilde\Delta_2  \rangle $ respectivement 
telles qu'on ait un diagramme commutatif~:
$$\xymatrix{
 \ll_2 \times \ll_1 \ar[r]^ {\Jj_{\(\tilde\Delta_1,\tilde\Delta_2\)}|_{\ll_2 
     \times \ll_1 }} \ar[d]_{\otimes \flb}& \ll'_1 \times \ll'_2 \ar[d]^{\otimes\flb} \\ 
  \left<  \Delta_2  \right> \times  \left<  \Delta_1 \right> \ar[r]_{ 
   \Jj_{\(\Delta_1,\Delta_2\)}} &  \left<  \Delta_1 
    \right> \times  \left<  \Delta_2 \right>  
}$$ 
où $\otimes\flb$ désigne le foncteur d'extension des scalaires de $\zlb$ à 
$\flb$. 
\end{prop}

\begin{proof}
Notons $\a$ la projection de 
$\left< \Delta_2 \right> \times \left< \Delta_1 \right>$ sur 
$\left< \Delta_1 ,\Delta_2 \right> $
définie par la re\-mar\-que \ref{Z12}
et $\b$ le morphisme injectif de 
$ \left< \Delta_1 ,\Delta_2 \right>$ dans 
$\left< \Delta_1 \right> \times \left< \Delta_2 \right> $ 
défini par le théorème \ref{Z2}. 
On définit $\tilde{\a}$ et $\tilde{\b}$ de façon analogue.
D'après \cite[I.9.3]{Vig1}, l'image de $ \ll_2 \times \ll_1$ par $\tilde{\a}$ est une 
structure entière dans $ \langle \tilde\Delta_1 ,\tilde\Delta_2 \rangle$, 
notée $\ll$. 
D'après la proposition \ref{red2seg}, les $\flb$-représentations 
$\ll \otimes \flb$ et $ \left< \Delta_1 ,\Delta_2 \right>$ sont isomorphes. 
On note $m$ le degré de $\Delta_1+\Delta_2$ et on pose 
$\G=\G_m$. 

D'après \cite[II.4.7]{Vig1} (voir aussi la preuve du lemme 6.11 dans 
\cite{Dat2}), 
il existe une extension finie $\E/\mathbf{Q}_\ell$ telle que les 
$\qlb$-représentations $ \langle \tilde\Delta_1 \rangle$, 
$ \langle \tilde\Delta_2 \rangle$ et 
$\langle \tilde\Delta_1,\tilde\Delta_2 \rangle$
admettent des modèles sur $\E$, notés respectivement 
$\pi_1^{\E},\pi_2^{\E}$ et $\pi^{\E}$. 
On peut supposer que $\ll$ est de la forme $\zlb\otimes \mathfrak{v}'$ 
où $\mathfrak{v}'$ est une $\Oo_\E$-structure entière de $\pi^\E$. 
Soient $\mathfrak{v}_1$ et $\mathfrak{v}_2$ deux $\Oo_\E$-structures 
entières quelconques de $\pi_1^\E$ et $\pi_2^\E$ respectivement. 
Par \cite[I.9.3]{Vig1}, l'image réciproque 
$\tilde{\b}^{-1}(\mathfrak{v}_1\times\mathfrak{v}_2)$ est une 
$\Oo_\E$-structure entière $\mathfrak{v}$ de $\pi^\E$. 
Puisque $\mathfrak{v}'$ et $\mathfrak{v}$ sont de type fini en tant que 
$\Oo_\E\G$-modules, il existe une constante $a\in\E$ 
telle que 
$\mathfrak{v}'\subseteq a \mathfrak{v}$ et $\mathfrak{v}'\not\subseteq 
a\p_{\E} \mathfrak{v}$.  
Soient $\ll'_1= \zlb\otimes a\mathfrak{v}_1$ et 
$\ll'_2= \zlb\otimes a\mathfrak{v}_2$. 
Ce sont des structures entières dans $ \langle \tilde\Delta_1 \rangle $ et 
$ \langle \tilde\Delta_2 \rangle $ respectivement. 
Par construction, le diagramme~: 
$$\xymatrix{
\ll\ar[r]^-{\tilde{\b} |_{\ll }}\ar[d]_{\otimes \flb}&\ll'_1\times\ll'_2 
\ar[d]^{\otimes \flb}\\ 
\left<\Delta_1,\Delta_2\right>\ar[r]_-{\b}&\left< 
 \Delta_1 \right> \times \left< \Delta_2 \right>}$$ 
est commutatif, ce qui achève la preuve de la proposition. 
\end{proof}

\subsubsection{}
\label{cups} 

Soit $\sss \in \Dive(\Cusp_{\flb})$ un support connexe banal et supposons 
que~: 
\begin{equation*}
\supp(\sss)=\sy{\rho}+\sy{\rho\mu_{\rho}}+\dots+\sy{\rho\mu_{\rho}^t},
\quad
0\< t\< e(\rho)-1. 
\end{equation*}
Soient $\m \in \MS(\sss)$ un 
multisegment banal et $\( \Delta_1,\dots,\Delta_r \)$ une forme rangée de 
$\m$.  
Pour tout $1 \< i \< r$, on pose $\Delta_i= \left[a_i,b_i\right]_{\rho}$ avec 
$0 \< a_i \< b_i \< t$.  
Pour chaque entier $i$, on pose $\tilde\Delta_i= \left[a_i,b_i \right]_{\tilde\rho}$, 
puis on pose $\tilde\m=\tilde\Delta_1+\dots+\tilde\Delta_r$.  
La condition $0 \< a_i \< b_i \< t$ assure que,
pour tous $1 \< i,j \< r$, on a~:
$$\tilde\Delta_i \text{ précède } \tilde\Delta_j \text{ si et seulement si 
} \Delta_i \text{ précède } \Delta_j.$$ 
On en déduit le corollaire suivant. 

\begin{coro}
\label{red-con} 
Avec les notations précédentes, pour 
$1 \<i \< r$, soit $\ll_i$ une structure entière dans 
$\langle \tilde\Delta_{i}\rangle$.  
Alors il existe, pour tout $1 \< i \< r$, une structure entière $\ll'_i$ dans 
$\langle  \tilde\Delta_{i}  \rangle $ telle qu'on ait un diagramme commutatif~:
$$\xymatrix{ 
 \ll_r \times \ll_{r-1} \times \dots \times \ll_1 \ar[rrr]^ {\Jj_{\tilde\m}|_{ \ll_r 
     \times \ll_{r-1} \times \dots \times \ll_1 }} \ar[d]_{\otimes \flb}&&& \ll'_1 
 \times \dots \times \ll'_{r-1} \times \ll'_r \ar[d]^{\otimes \flb} \\ 
  \I\left(\Delta_r, \dots, \Delta_1\right)  \ar[rrr]_{\Jj_{\m} }& &&  \I\left(\Delta_1, \dots, \Delta_r \right)
}$$
où $\otimes\flb$ désigne le foncteur d'extension des scalaires de $\zlb$ à 
$\flb$. 
\end{coro}

\begin{proof}
D'après le lemme \ref{Tadmult} et la proposition \ref{enti}, il ne nous 
reste à voir que, si $ \Delta, \Delta' $ sont deux segments banals 
non liés et $\ll$ et $\ll'$ sont deux structures entières dans 
$ \langle  \tilde\Delta  \rangle $ et $ \langle \tilde\Delta^{\prime}  \rangle $ respectivement, alors 
le diagramme~: 
$$\xymatrix{ 
 \ll \times \ll' \ar[r]^ {\Jj_{\(\tilde\Delta,\tilde\Delta^{\prime}\)}|_{\ll \times \ll' 
   }} \ar[d]_{\otimes \flb}& \ll' \times \ll \ar[d]^{\otimes \flb} \\ 
  \left<  \Delta  \right> \times  \left<  \Delta'  \right> \ar[r]_{ 
   \Jj_{\(\Delta,\Delta'\)}} &  \left< \Delta' 
    \right> \times \left< \Delta \right> 
}$$
est toujours commutatif.  
Or, les flèches horizontales étant, par le théorème \ref{nuevo2}, des 
iso\-mor\-phismes, le résultat est clair.  
\end{proof}

\subsubsection{} 

On prouve maintenant le théorème \ref{releverbanales}.
Par la proposition \ref{despegados2}, on se ramène au cas où 
la représentation $\pi$ est de support cuspidal $\sss$ connexe. 
On peut maintenant appliquer le corollaire \ref{red-con}. 
Soit $\m \in \MS(\sss)$ un multisegment banal tel que $\pi=\langle\m\rangle$ 
et soit $\( \Delta_1,\dots,\Delta_r \)$ une forme rangée de $\m$. 
On définit $\tilde\m$ comme dans \ref{cups}. 
Pour $1 \< i \< r$, soient $\ll^{}_i, \ll'_i$ des struc\-tures entières dans 
$ \langle  \tilde\Delta_i \rangle $ telles qu'on ait un diagramme commutatif~: 
$$\xymatrix{ 
 \ll_r \times \ll_{r-1} \times \dots \times \ll_1 \ar[rrr]^ {\Jj_{\tilde\m}|_{ \ll_r 
     \times \ll_{r-1} \times \dots \times \ll_1 }} \ar[d]_{\otimes \flb}&&& \ll'_1 
 \times \dots \times \ll'_{r-1} \times \ll'_r \ar[d]^{\otimes \flb} \\ 
  \I\left(\Delta_r, \dots, \Delta_1\right)  \ar[rrr]_{\Jj_{\m} }& &&  \I\left(\Delta_1, \dots, \Delta_r \right) . 
}$$
On en déduit que~: 
$$\Jj_{\m} \( \( \ll_r \times \ll_{r-1} \times \dots \times \ll_1 \) \otimes 
\flb \) \simeq \Jj_{\m} \( \I\left(\Delta_r, \dots, \Delta_1\right) \) \simeq 
\pi.$$ 
D'un autre côté, l'image de $ \ll_r \times \ll_{r-1} \times \dots \times 
\ll_1$ par $\Jj_{\tilde{\m}}$ est une structure entière $\ll$ de la
représentation 
$\tilde\pi= \langle \tilde\m \rangle $ par le lemme \ref{Tadmult}, 
et donc la 
commutativité du diagramme implique que~: 
$$\ll \otimes \flb\simeq\pi,$$ 
c'est-à-dire $\r_{\ell}(\tilde\pi)=\sy{\pi}$, ce qui achève la preuve du théorème. 

\begin{rema}
Ainsi, pour relever une représentation irréductible banale $\pi$ 
de la forme $\langle\m\rangle$ pour un 
multisegment $\m=\Delta_1+\dots+\Delta_r $ 
en une représentation $\tilde\pi$ de la forme $\langle\tilde\m\rangle$ 
pour un multisegment 
$\tilde\m=\tilde\Delta_1+\dots + \tilde\Delta_r $, il suffit que les segments 
$\tilde\Delta_{i}$ soient des relèvements des seg\-ments $\Delta_i$ 
tels que, pour tous $1 \< i,j \< r$, on ait~: 
$$ \tilde\Delta_i \text{ précède } \tilde\Delta_j \text{ si et seulement si 
} \Delta_i \text{ précède } \Delta_j .$$ 
L'hypothèse de banalité sur $\pi$ nous assure qu'un tel choix des 
$\tilde\Delta_{i}$ est possible. 
\end{rema}

\providecommand{\bysame}{\leavevmode ---\ }
\providecommand{\og}{``}
\providecommand{\fg}{''}
\providecommand{\smfandname}{\&}
\providecommand{\smfedsname}{\'eds.}
\providecommand{\smfedname}{\'ed.}
\providecommand{\smfmastersthesisname}{M\'emoire}
\providecommand{\smfphdthesisname}{Th\`ese}


\providecommand{\bysame}{\leavevmode ---\ }
\providecommand{\og}{``}
\providecommand{\fg}{''}
\providecommand{\smfandname}{\&}
\providecommand{\smfedsname}{\'eds.}
\providecommand{\smfedname}{\'ed.}
\providecommand{\smfmastersthesisname}{M\'emoire}
\providecommand{\smfphdthesisname}{Th\`ese}
\begin{thebibliography}{}

\end{thebibliography}


\begin{thebibliography}{10}

\bibitem{ArikiBook}
{\scshape S.~Ariki} -- \emph{Representations of quantum algebras and
  combinatorics of {Y}oung tableaux}, University Lecture Series, vol.~26,
  American Mathematical Society, Providence, RI, 2002. 

\bibitem{Auel}
{\scshape A.~N. Auel} -- \emph{Une d\'emonstration d'un th\'eor\`eme de
  {B}ernstein sur les repr\'esentations de quasi carr\'e int\'egrable de ${{\rm
  {G}{L}}}_{n}({F})$ o\`u ${F}$ est un corps local non archim\'edien},
  M\'emoire de DEA, Universit\'e Paris Sud, 2004.

\bibitem{BHLS}
{\scshape A.~I. Badulescu, G.~Henniart, B.~Lemaire {\normalfont \smfandname}
  V.~S{\'e}cherre} -- {\og Sur le dual unitaire de {${\rm GL}_r(D)$}\fg},
  \emph{Amer. J. Math.} \textbf{132} (2010), no.~5, p.~1365--1396.

\bibitem{Bad}
{\scshape A.~I. Badulescu} -- {\og Un r\'esultat d'irr\'eductibilit\'e en
  caract\'eristique non nulle\fg}, \emph{Tohoku Math. J. (2)} \textbf{56}
  (2004), no.~4, p.~583--592.

\bibitem{BZ1}
{\scshape I.~N. Bernstein {\normalfont \smfandname} A.~V. Zelevinsky} -- {\og
  Representations of the group {$GL(n,F),$} where {$F$} is a local
  non-{A}rchimedean field\fg}, \emph{Uspehi Mat. Nauk} \textbf{31} (1976),
  no.~3(189), p.~5--70.

\bibitem{BZ2}
\bysame , {\og Induced representations of reductive {$p$}-adic groups. {I}\fg},
  \emph{Ann. Sci. \'Ecole Norm. Sup. (4)} \textbf{10} (1977), no.~4,
  p.~441--472.

\bibitem{CG}
{\scshape N.~Chriss {\normalfont \smfandname} V.~Ginzburg} --
  \emph{Representation theory and complex geometry}, Birkh\"auser Boston Inc.,
  Boston, MA, 1997.

\bibitem{Dat2}
\bysame , {\og {$\nu$}-tempered representations of {$p$}-adic groups, {I}: 
  {$l$}-adic case\fg}, \emph{Duke Math. J.} \textbf{126} (2005), no.~3,
  p.~397--469.

\bibitem{Dat3}
{\scshape J.-F. Dat} -- {\og Finitude pour les repr\'esentations lisses de
  groupes {$p$}-adiques\fg}, \emph{J. Inst. Math. Jussieu} \textbf{8} (2009),
  no.~2, p.~261--333.

\bibitem{DKV}
{\scshape P.~Deligne, D.~Kazhdan {\normalfont \smfandname} M.-F. Vign{\'e}ras}
  -- {\og Repr\'esentations des alg\`ebres centrales simples {$p$}-adiques\fg},
  in \emph{Representations of reductive groups over a local field}, Travaux en
  Cours, Hermann, Paris, 1984, p.~33--117.

\bibitem{MS}
{\scshape A.~M{\'{\i}}nguez {\normalfont \smfandname} V.~S{\'e}cherre} -- {\og
  Repr{\'e}sentations lisses modulo {$\ell$} de {${\rm GL}_m(\D)$}\fg},
  Prépu\-blication \verb!arXiv:1110.1467!.

\bibitem{Rodier}
{\scshape F.~Rodier} -- {\og Repr\'esentations de {${\rm GL}(n,\,k)$} o\`u
  {$k$} est un corps {$p$}-adique\fg}, in \emph{Bourbaki Seminar, Vol.
  1981/1982}, Ast\'erisque, vol.~92, Soc. Math. France, Paris, 1982,
  p.~201--218.

\bibitem{Tadic}
{\scshape M.~Tadi{\'c}} -- {\og Induced representations of {${\rm GL}(n,A)$}
  for {$p$}-adic division algebras {$A$}\fg}, \emph{J. Reine Angew. Math.}
  \textbf{405} (1990), p.~48--77.

\bibitem{Vig1}
{\scshape M.-F. Vign{\'e}ras} -- \emph{Repr\'esentations {$l$}-modulaires d'un
  groupe r\'eductif {$p$}-adique avec {$l\ne p$}}, Progress in Mathematics,
  vol. 137, Birkh\"auser Boston Inc., Boston, MA, 1996.

\bibitem{Vi5}
\bysame , {\og Irreducible modular representations of a reductive {$p$}-adic
  group and simple modules for {H}ecke algebras\fg}, in \emph{European
  {C}ongress of {M}athematics, {V}ol. {I} ({B}arcelona, 2000)}, Progr. Math.,
  vol. 201, Birkh\"auser, Basel, 2001, p.~117--133.

\bibitem{VigW}
\bysame , {\og On highest Whittaker models and integral structures}, 
 in \emph{Contributions to Automorphic forms, Geometry and Number theory: 
   Shalikafest 2002}, John Hopkins Univ. Press, 2004, p.~773--801.

\bibitem{Ze2}
{\scshape A.~V. Zelevinsky} -- {\og Induced representations of reductive
  {${\mathfrak{p}}$}-adic groups. {II}. {O}n irreducible representations of
  {${\rm GL}(n)$}\fg}, \emph{Ann. Sci. \'Ecole Norm. Sup. (4)} \textbf{13}
  (1980), no.~2, p.~165--210.

\end{thebibliography}
\end{document}